\documentclass[12pt]{article}
\usepackage{latexsym,amssymb,amsmath,a4wide,color,graphicx,leftindex}
\usepackage{stackengine}

\newtheorem{proposition}{Proposition}
\newtheorem{theorem}{Theorem}
\newtheorem{corollary}{Corollary}

\newcommand{\PSL}{{\rm PSL}}
\newcommand{\PGam}{{\rm P\Gamma L}}
\newcommand{\GF}{{\rm GF}}
\newcommand{\Aff}{{\rm Aff}}
\newcommand{\Aut}{{\rm Aut}}
\newcommand{\dia}{{\rm dia}}
\newcommand{\off}{{\rm off}}
\newcommand{\upp}{{\rm upp}}
\newcommand{\ord}{{\rm ord}}

\newcommand{\veps}{\varepsilon}

\newcommand{\m}{\kern 0.06em}
\newcommand{\ovl}{\overline}
\newcommand{\lcm}{{\rm lcm}}

\newcommand{\lfix}{\leftindex_{\rm fix}}

\newcommand{\is}{{\rm iso}}
\newcommand{\ho}{{\rm hom}}
\newcommand{\ep}{{\rm epi}}

\newcommand\xrowht[2][0]{\addstackgap[.5\dimexpr#2\relax]{\vphantom{#1}}}

\begin{document}

\title{\vspace{-3cm} Enumeration of orientably-regular maps \\
with automorphism group $\PGam(2,2^p)$ for prime $p$  }

\author{}
\date{}
\maketitle

\begin{center}
\vspace{-1.7cm}

{\large Yan-Quan Feng} \\
\vspace{0.8mm} {\small Beijing Jiaotong University, Beijing, P. R. China}\\
\vspace{2.7mm} {\large So\v{n}a Pavl\'ikov\'a, Jozef \v Sir\'a\v n}\\
\vspace{0.8mm} {\small Slovak University of Technology, Bratislava, Slovakia}\\
\vspace{2.7mm} {\large Chihao Wang and Jin-Xin Zhou} \\
\vspace{0.8mm} {\small Beijing Jiaotong University, Beijing, P. R. China}
\vspace{3mm}

\end{center}

\begin{abstract}
With help of group characters we enumerate orientably-regular maps of a given type with automorphism group isomorphic to $\PGam(2,2^p)$ for prime $p$.
\end{abstract}

\vskip 3mm


\section{Introduction}\label{sec:intro}

\

A {\em map} is a cellular embedding of a connected graph on an orientable surface; the embedded graph is the {\em underlying graph} of the map. We will assume that the {\em carrier surface} of the map has been given a specific orientation, making the surface and the map {\em oriented}. We will also restrict ourselves to underlying graphs in which every vertex has the same valency $n\ge 3$. The $2$-cells of the embedding are the {\em faces} of the map. Every face is bounded by a closed walk in the underlying graph. Along with assuming every vertex of valency $n$, we will also restrict ourselves to considering maps in which every face boundary walk has the same length $m\ge 3$, constituting {\em uniform} maps, of {\em type} $\{m,n\}$.
\smallskip

An {\em automorphism} of a map $M$ of type $\{m,n\}$ for $m,n\ge 3$ is an automorphism of the underlying graph of $M$ which preserves face boundaries oriented consistently with the chosen orientation of the carrier surface of $M$. The group $\Aut(M)$ of all automorphisms of $M$ acts freely (or, in an equivalent terminology, semi-regularly) on the set of arcs (that is, vertex-edge incident pairs) of the underlying graph; in the case this action is transitive (and hence regular) we say that the map $M$ is {\em orientably-regular}.
\smallskip

By general theory of maps on orientable surfaces \cite{JoSi}, an {\em orientably-regular map} $M$ of type $\{m,n\}$ may be identified with a presentation of its automorphism group $G=\Aut(M)$ generated by a pair of distinct elements $x,y$, of respective orders $m$ and $n$, with product of order $2$; in symbols,
\begin{equation}\label{eq:Aut(M)}
G = \langle x,\,y\,;\, x^m,\,y^n,\,(xy)^2,\ldots \rangle \ .
\end{equation}
The automorphisms $x$ and $y$ may be interpreted, respectively, as an $m$-fold rotation of $M$ about the centre of some fixed face $\Phi$ of $M$ and an $n$-fold rotation about some fixed vertex $v$ on the boundary walk of $\Phi$, and then $xy$ represents a two-fold rotation about the centre of an edge incident to $v$ and lying on the boundary walk of $\Phi$; the orientation of all the three rotations are assumed to be induced by the chosen orientation of the carrier surface of $M$. An orientably-regular map $M$ of type $\{m,n\}$ determined by a presentation of a group $G$ as in \eqref{eq:Aut(M)} will be denoted by ${\rm Map}(G;\,x,y)$, or simply by ${\rm Map}(x,y)$ if the group $G\cong \Aut(M)$ is clear from the context.
\smallskip

Invoking \cite{JoSi} again, or referring to a recent survey \cite{Si-surv}, two orientably-regular maps $M=(G;\,x,y)$ and $\ovl{M}=(\ovl{G};\,\ovl{x},\ovl{y})$ are isomorphic if and only if there is a group isomorphism $G\to \ovl{G}$ taking $x$ to $\ovl{x}$ and $y$ to $\ovl{y}$. In particular, isomorphic maps have the same type; the converse is, of course, not true in general.
\smallskip

The presentation \eqref{eq:Aut(M)} for the group $G=\Aut(M)$ of an orientably-regular map $M$ implies that $G$ is a quotient of the well-known $(2,m,n)$-{\em triangle group} $\Delta(2,m,n)$ with presentation \begin{equation}\label{eq:triangle} \Delta(2,m,n) = \langle X,Y\,;\, X^m,\,Y^n,\,(XY)^2\rangle \end{equation} by some torsion-free normal subgroup $K$. Parallel to this, the presentation \eqref{eq:Aut(M)} also implies that the map itself can be regarded as a quotient, by the same subgroup $K$, of a tessellation of a simply-connected surface (a sphere, an Euclidean plane or a hyperbolic plane, depending on whether $1/m + 1/n$ is larger than, equal to or smaller than $1/2$) by mutually congruent $m$-gons, $n$ of which meet at every vertex. Both the natural epimorphism $\Delta(2,m,n)\to G$ with kernel $K$, as well as (surprisingly) the epimorphism from $G$ onto the trivial group, lead to numerous deep connections between regular maps, hyperbolic geometry, complex functions encoded in terms of so-called `dessins d'enfant', and Galois theory; see e.g. \cite{JoS1} and \cite{Jones} for surveys on these topics.
\smallskip

Broadly, investigation of orientably-regular maps is approached in three main directions, depending on whether the focus is on a given underlying graph, a given carrier surface, or a given automorphism group. Here we will consider orientably-regular maps with a given automorphism group, with emphasis on their enumeration. It appears that the only non-trivial infinite families of groups $G$ for which enumeration results for orientably-regular maps $M$ with $\Aut(M)\cong G$ are available, are the linear fractional groups $\PSL(2,q)$ and ${\rm PGL} (2,q)$ for prime powers $q$. In more detail, enumeration results of orientably-regular maps $M$ of unspecified type with $\Aut(M)\cong \PSL(2,q)$ have been stated in \cite{Dow} and explicitly elaborated for $q$ a power of $2$ in \cite{DJ}, and enumeration formulae for such maps but with a specified type have been announced in \cite{Adr}, preceded by classification of orientably-regular maps $M$ of a given type with automorphism group isomorphic to $\PSL(2,q)$ or ${\rm PGL}(2,q)$ in \cite{Sah}, from which the enumeration results of \cite{Adr} and \cite{DJ} can be extracted.
\smallskip

The aim of this paper is to extend the results of \cite{Adr} and \cite{DJ} to enumeration of orientably-regular maps $M$ with $\Aut(M)$ isomorphic to $G=\PGam(2,q)$ for $q=2^p$, where $p$ is an odd prime, up to map isomorphism. Two kinds of such enumeration results will be presented: one for maps of unspecified type, obtained by Hall's enumeration scheme \cite{Hal} with help of the M\"obius function for the lattice of subgroups of $G$, and another one for maps of specified type, using the character table for $G$ and a formula due to Frobenius \cite{Fr} for counting pairs of generators $G$ of given order and contained in specified conjugacy classes of $G$.
\smallskip

The text is organised as follows. In Section \ref{sec:strat} we give details concerning the two enumeration strategies of \cite{Fr} and \cite{Hal}. Sections \ref{sec:e} and \ref{sec:p} contain useful information regarding the group $\PGam(2,2^e)$, first for arbitrary $e\ge 1$ and then for $e$ an odd prime, focusing on orders of elements and conjugacy. In Sections \ref{sec:Mob} and \ref{sec:enum} we derive the M\"obius function for the group $G=\PGam(2,2^p)$ for an odd prime $p$ and prove our first main result on enumeration of isomorphism classes of orientably-regular maps $M$ with $\Aut(M)\cong \PGam(2,2^p)$ of unspecified type. Continuing in Section \ref{sec:char} with a derivation of the character table of $G$, in Section \ref{sec:enum-spec} we prove our second main result on enumeration of isomorphism classes of orientably-regular maps $M$ {\em of specified type} with $\Aut(M)\cong \PGam(2,2^p)$. In Section \ref{sec:invariance} we discuss invariance of these maps under map operators, and conclude with remarks in the final Section \ref{sec:rem}.
\smallskip


\section{Strategies}\label{sec:strat}

\

Enumeration of orientably-regular maps with a given automorphism group but with unspecified type turns out to be a special case of enumeration of isomorphism classes of regular objects with given finite automorphism group, which  was developed as early as in 1936 by P. Hall \cite{Hal}; for a comprehensive summary of Hall's method we refer to \cite{DJ}. For orientably-regular maps $M$ with a given $G=\Aut(M)$ presented as in \eqref{eq:Aut(M)} with unspecified $m$ and $n$ it will be of advantage to regard $G$ as a quotient of a generalised version of triangle groups $\Delta(m,n)$ in which one allows both $m$ and $n$ to be equal to infinity. A presentation of such a generalisation $\Delta= \Delta(2,\infty, \infty)$ may be obtained from \eqref{eq:triangle} by leaving out the relators $X^m$ and $Y^n$, but it will be more handy to work with an equivalent presentation of this group, of the form $\Delta = \Delta(2,\infty,\infty) = \langle T,X\,;\, T^2=1\rangle$, from which it is clear that $\Delta$ is isomorphic to the free product $C_2 * C_{\infty}$ of the cyclic groups $\langle T\rangle \cong C_2$ and $\langle X\rangle \cong C_\infty$.
\smallskip

It follows that, for our specific task of enumeration of orientably-regular maps with automorphism group isomorphic to the group $G=\PGam(2,2^p)$ for prime $p\ge 5$, Hall's method assumes to perform the following steps.
\medskip

{\bf Step 1.} Determining the M\"obius function $\mu$ for the lattice of subgroups of the group $G=\PGam(2,q)$.  Values of this function at an arbitrary subgroup of $G$ are given by
\begin{equation}\label{eq:Mob} \mu(G)=1\ \ {\rm and} \ \ \sum_{J\ge K} \mu(J) = 0 \ \ {\rm for\ every\  proper\ subgroup}\ K{<}G\ \ .\end{equation}
A useful hint for calculating values of the M\"obius function of a lattice of subgroups of a group $G$ was pointed out in \cite[Theorem 2.3]{Hal}, by which $\mu(K) = 0$ whenever $K$ is {\em not} an intersection of a set of maximal subgroups of $G$. It follows that {\sl Step 1} may be restricted to subgroups $K{<}G$ which arise as intersections of maximal subgroups of $G$.
\smallskip

{\bf Step 2.} Since $\Delta = \Delta(2,\infty,\infty)$ is generated by two elements, $T$, of order $2$ and $X$ of an infinite order, every homomorphism $\Delta \to K$ is uniquely determined by assigning to $T$ an element of order $\le 2$ of $K$, and an arbitrary element of $K$ to $X$. The number $\ho(K)$ of homomorphisms from $\Delta$ to $K$ is therefore given by $\ho(K) = |\{z\in K\,;\,z^2=1\}|{\cdot}|K|$. If $\ep(K)$ denotes the number of epimorphisms $\Delta\to K$ for an arbitrary subgroup $K\le G$, one has $\ho(K) = \sum_{L\le K}\ep(L)$. This set of equations over the lattice of subgroups $L$ of $K$ for $K\le G$ can be inverted with the help of the M\"obius function $\mu_G$ from {\sl Step 1}, by which one obtains
\begin{equation}\label{eq:inv} \ep(G) = \sum_{K\le G}\mu(K){\cdot} \ho(K) \ \ .\end{equation}
The quantity $\ep(G)$ from \eqref{eq:inv} is equal to the number of {\em generating pairs} $(t,x)$ of $G$ satisfying $t^2=x^m=1$ for some $m\ge 1$. Observe that, letting $y=x^{-1}t$ (so that $t=xy$) and letting $n$ be the order of $y$, it follows that $\ep(G)$ is equal to the number of generating pairs $(x,y)$ of $G$ satisfying $x^m=y^n=(xy)^2=1$ for some $m$ and $n$ as in the presentation \eqref{eq:Aut(M)}.
\smallskip

{\bf Step 3.} Finally, the number ${\rm map}(G)$ of orientably-regular maps with automorphism group isomorphic to $G$, up to map isomorphism, is equal to the number of {\em orbits} of generating pairs $(x,y)$ for $G$ determined in {\sl Step 2}, arising by the action of the automorphism group of $G$ on such pairs. For our group $G = \PGam(2,2^p) \cong \PSL(2,2^p)\rtimes C_p$ admitting only the trivial homomorphism from $C_p$ onto the (trivial) centre of $\PSL(2,2^p)$ for $p\ge 2$, it follows from \cite[Theorem 1]{Cur} that $\Aut(G) \cong \Aut(H) \cong G$. The number ${\rm map}(G)$ is then simply given by ${\rm map}(G) = \ep(G)/|\Aut(G)|$, where $|\Aut(G)| = |G|$.
\medskip

For enumeration of orientably-regular maps with a given automorphism group but also of a given type $\{m,n\}$ we use a different strategy that begins with determining the character table of $G=\PGam(2,2^p)$. This is necessary for a subsequent application of a formula of Frobenius \cite{Fr} which, with help of characters of $G$, enables one to count pairs of generators $G$ of order $m$ and $n$, with product of order $2$, contained in preassigned conjugacy classes of $G$. The stages leading towards this kind of enumeration aim are as follows.
\medskip

{\bf Stage 1.} Deriving a character table for $G$. For this one needs to determine conjugacy classes of $G = \PGam(2,2^p)$, which turn out to have manageable structure due to primality of $p$, especially regarding those lying outside the subgroup $H=\PSL(2,2^p)$ of $G$; those contained inside $H$ are obtained by merging (in $G$) the known conjugacy classes of $\PSL(2,2^p)$ in a controllable way. Most entries in the character table for $G$ turn out to be obtainable by inducing characters from $H$ and lifting characters from $G/H$, but some need using  Clifford's theory and Brauer's characterisation of irreducible characters as integer class functions which restrict to generalised characters on so-called elementary subgroups (all to be introduced in Section \ref{sec:char}).
\smallskip

{\bf Stage 2.} By the Frobenius formula of \cite{Fr} (see also \cite{Jon} for an explanation and applications), for a finite group $G$ the number $\#({\cal A},{\cal B},{\cal C})$ of ordered triples $(x,y,z)$ of elements of $G$  satisfying $xyz=1$ and such that $x,y,z$ belong, respectively, to the conjugacy classes ${\cal A},{\cal B},{\cal C}$ in $G$, is given by
\begin{equation}\label{eq:Fr}
\#({\cal A},{\cal B},{\cal C})= \frac{|{\cal A}||{\cal B}||{\cal C}|}{|G|}\sum_{\chi} \frac{\chi({\cal A})\chi({\cal B})\chi({\cal C})}{\chi(id)} \end{equation}
where the summation ranges over irreducible characters of $G$ and $id$ is the (conjugacy class of) the identity element of $G$. For the group $G=\PGam(2,2^p)$ and orientably-regular maps $M$ of a given type $\{m,n\}$ with $\Aut(M)\cong G$ one needs to apply \eqref{eq:Fr} to all combinations of conjugacy classes ${\cal A}$, ${\cal B}$ and ${\cal C}$ of elements of order $m$, $n$ and $2$ in $G$, respectively.
\smallskip

{\bf Stage 3.} Among the pairs $(x,y)$ of elements of $G=\PGam(2,2^p)$ arising from triples $(x,y,z)$ identified in {\sl Stage 2}, that is, such that $x$ has order $m$, $y$ has order $n$ and $z=(xy)^{-1}=xy$ has order $2$, one needs to eliminate the pairs that {\em do not generate} the entire group $G$. Alternatively, using the findings from {\sl Stage 2} and having determined the number of {\em smooth} homomorphisms $\Delta(X,Y)\to G=\langle x,y\rangle$, that is, homomorphisms preserving orders of elements in the two generating pairs, one may determine the number of smooth homomorphisms $\Delta(X,Y)\to K$ for every proper subgroup $K<G$ with a non-zero value of the M\"obius function and use \eqref{eq:inv} to determine the number generating pairs of $G$ of specified orders; we will use this alternative approach in Section \ref{sec:enum-spec}. The number of orientably-regular maps $M$ of type $\{m,n\}$ with $\Aut(M)\cong G$ is then obtained by dividing the number of {\em generating pairs} $(x,y)$ of $G$ arising from {\sl Stage 2} by the order $|G|$ of the automorphism group of $G$ as in {\sl Step 3} of the strategy for enumeration of maps of unspecified type.
\medskip

Both kinds of enumeration processes described in this section require knowledge about orders of elements in the group $G=\PGam(2,2^p)$, their conjugacy, and also about conjugacy of subgroups of $G$; collecting information on these forms the contents of the forthcoming Sections \ref{sec:e} and \ref{sec:p}.
\smallskip


\section{The group ${\bf{P\Gamma L}}({\bf{2}},{\bf{2}}^{\bf{\textit e}})$: a few basic facts}\label{sec:e}

\

We begin with considering the more general case and let $q=2^e$ for an arbitrary positive integer $e$. Recall that $\PGam(2,q) = \PSL(2,q)\rtimes C_e$, a semi-direct product of $\PSL(2,q)={\rm SL}(2,q)$ by the group of automorphisms of the finite field $F=\GF(q)$ isomorphic to the cyclic group of order $e$, generated by the Frobenius automorphism $\theta:\ z\mapsto z^2$ for $z\in F$. Elements of the group $H=\PSL(2,q)$ will be identified with $2\times 2$ matrices over $F$ with determinant $1$; note that $|H|=q(q^2{-}1)$. Elements of the group $G=\PGam(2,q) \cong H\rtimes C_e$ are pairs $(A,j)$ for $A\in H$ and $j\in C_e = \{0,1,\ldots,e{-}1\}$, with multiplication given by the rule
\begin{equation}\label{eq:multip}
(A,j)(B,k) = (AB^{[j]},j{+}k)\
\end{equation}
where the exponent $[j]$ is a shorthand for $j$-fold application of the Frobenius automorphism $\theta$ to every element of the $2\times 2$ matrix $B$ over $F$, which means that $B^{[j]}$ is obtained from $B$ by raising every its entry to the power of $2^j$.
\smallskip

The group $H=\PSL(2,q)$ trivially embeds in $\PSL(2,q^2)$, but an analogue of this inclusion for $\PGam(2,q)$ is valid only for {\em odd} $e$:
\smallskip

\begin{proposition}\label{prop:embed}
The group $\PGam(2,2^e)$ embeds in $\PGam(2,2^{2e})$ if and only if $e$ is odd.
\end{proposition}

{\bf Proof.} For odd $e$, multiplication by $2$ induces an automorphism of $C_e$. Let $(2i)_e$ and $(2i)_{2e}$ be the residue class of $2i$ (mod $e$) and (mod $2e$), respectively. Then, the assignment $\veps:\ (2i)_e\mapsto (2i)_{2e}$ for $i\in \{0,1,\ldots, e{-}1\}$ defines an embedding $C_e\to C_{2e}$. It may be checked that for every $j\in \{0,1,\ldots, e{-}1\}$ the injective homomorphism $\veps$ has the property that $\veps(j)=j$ if $j$ is even, while $\veps(j)=j{+}e$ if $j$ is odd, interpreting in both cases the argument for $\veps$ and the image of $\veps$ as a residue class (mod $e$) and (mod $2e$), respectively. It follows that $\veps$ may equally well be defined by $\veps(i) = ti$ (mod $2e$) for $t=e{+}1$.
\smallskip

Thus, for $j\in \{0,1,\ldots, e-1\}$ and $B\in \PSL(2,2^e)$, if $j$ is even one has $B^{[\veps(j)]}= B^{[j]}$, and if $j$ is odd, then $B^{[\veps(j)]}= B^{[j{+}e]}=B^{[j]}$ again, as $B^{[e]}=B$ for $B\in \PSL(2,2^e)$. If $e$ is odd, the mapping $\psi$ given by $(A,j)\mapsto (A,\veps(j))$ for every $A\in \PSL(2,2^e)$ and $j\in \{0,1, \ldots,e{-}1\}$ gives an injective homomorphism from $G=\PGam(2,2^e)$ into ${G^*}=\PGam(2,2^{2e})$. Indeed, checking that $\psi((A,j)(B,k))=\psi(A,j)\psi(B,k)$ for $j,k\in \{0,1,\ldots,e{-}1\}$ representing residue classes (mod $e$) as arguments of $\veps$, amounts to verifying that $\psi(AB^{[j]},j{+}k) = (A,\veps(j))(B,\veps(k)) = (AB^{[\veps(j)]}, \veps(j{+}k))$. But the latter boils down to the equation $B^{[j]} = B^{[\veps(j)]}$ established earlier.
\smallskip

Let us proceed by listing a few more facts. For an arbitrary $e\ge 1$ the group $H=\PSL(2,2^e)$ naturally embeds in $\PSL(2,2^{2e})$ and, moreover, $H$ is a normaliser of itself in $\PSL(2,2^{2e})$. By \cite[Table 8.1]{Bray}, the quotient group of the normaliser $N_{G^*}(H)/H$ is isomorphic to the cyclic group $C_{2e}$. Further, by  the well known Normaliser-Centraliser Theorem one concludes that $N_{G^*}(H)/C_{G^*}(H)$ embeds in $\Aut(H)\cong G$.
\smallskip

To finish by contradiction, suppose now that for some even $e\ge 2$ the group $G=\PGam(2,2^e)$ embeds in $G^* = \PGam(2,2^{2e})$. This implies that $N_{G^*}(H)=C_{G^*}(H)\times G$, and note that $C_{G^*}(H) \cap G=1$. But here $|C_{G^*}(H)|=2$, so that $G$ is a normal subgroup of $N_{G^*}(H)$ since the index of $G$ in $N_{G^*}(H)$ is $2$. It follows that $N_{G^*}(H)/H\cong C_2\times C_e$, which is not cyclic for even $e$, contrary to what has been established in \cite[Table 8.1]{Bray}.
\hfill $\Box$
\medskip

Subgroups of the group $H=\PSL(2,q)$ have been classified for an arbitrary prime power in \cite{Dic}, and as a reference we will use the comprehensive survey \cite{BDL}. For our purposes and $q=2^e$ we will only need more detailed information about three kinds of maximal subgroups of $H$ over the field $F=\GF(q)$; the affine subgroups and two kinds of dihedral subgroups.
\smallskip

First, $H$ contains $q+1$ mutually conjugate copies of an affine group ${\rm AGL}(1,q)$, of order $q(q-1)$. We will refer to a {\em standard copy} $\Aff_q<H$ of ${\rm AGL}(1,q)$, consisting of {\em upper triangular matrices} $\upp(\xi^i,b)$, in which the top off-diagonal entry $b$ ranges over elements of $F$ and the diagonal entries are powers $\xi^i$ and  $\xi^{-i}$ (from top to bottom) of a fixed primitive $(q{-}1)^{\rm st}$ root $\xi$ of unity in $F$ for $i\in \{1,2,\ldots, q{-}1\}$.
\smallskip

Second, $H$ contains $q(q+1)/2$ pairwise conjugate copies of dihedral subgroups of order $2(q{-}1)$. We will again refer to a particular copy, again called {\em standard} and denoted by $D_{2(q{-}1)}$, formed by $q{-}1$ {\em diagonal matrices} ${\rm dia}(\xi^i)$ and $q{-}1$ {\em off-diagonal matrices} ${\rm off}(\xi^i)$, in both cases with entries $\xi^i$, $\xi^{-i}$ (from top to bottom) for a fixed primitive $(q{-}1)^{\rm st}$ root $\xi$ of unity in $F$ and for $i\in \{1,2,\ldots, q{-}1\}$.
\smallskip

Third, $H$ contains $q(q-1)/2$ mutually conjugate copies of dihedral subgroups of order $2(q{+}1)$. Recalling the known approach to these subgroups (see e.g. \cite{CPS,Sah}), one may without loss of generality work with a particular isomorphic copy this group, denoted by $D_{2(q{+}1)}$, which is a subgroup of $\PSL(2,q^2)$ but {\em not} a subgroup of $H$. The group $D_{2(q{+}1)}$ is again formed by matrices of the form ${\rm dia}(\xi^i)$ and ${\rm off}(\xi^i)$ as in the previous case, but this time with $\xi$ a fixed  primitive $(q{+}1)^{\rm st}$ root of unity in the $2$-extension $\GF(q^2)$ of $F$, and with $i\in \{1,2,\ldots, q{+}1\}$.
\smallskip

Returning to the group $G=\PGam(2,q)$ for $q=2^e$, we will also need a few facts about orders of elements in $G$ lying outside the unique subgroup $H$ of $G$ isomorphic to $\PSL(2,q)$. Clearly, the order of $(A,0)\in G$ is the same as the order of $A$ in $H$. By \cite{BDL}, for the order $\ord(A)$ of an element $A\in H$ one has the following four cases:
\begin{itemize}
\setlength{\itemsep}{0pt}
\setlength{\parskip}{0pt}
\setlength{\parsep}{0pt}
\item[{\rm (a)}] $\ord(A)=1$ if and only if $A=I$, the identity matrix;
\item[{\rm (b)}] $\ord(A)=2$ if and only if $A\ne I$ and the trace ${\rm tr}(A)$ of $A$ is equal to $0$;
\item[{\rm (c)}] $\ord(A)$ is a divisor of $q{-}1$ if and only if $A$ is conjugate in $H$ to an element of the form ${\rm dia} (\xi^i)$ for some $\xi\in \GF(q)$ and $i\in\{1,2,...,q-2\}$ such that $\xi^{q{-}1}=1$;
\item[{\rm (d)}] $\ord(A)$ is a divisor of $q{+}1$ if and only if $A$ is conjugate, for some $\xi\in \GF(2^{2e})$ such that $\xi^{q{+}1}=1$, to an element of the form ${\rm dia} (\xi^i)\in \PSL(2,q^2)$, where $i\in \{1,2,...,q\}$.
\end{itemize}
As the next step, we will consider orders of elements of the form $(A,i)\in \PGam(2,q)$ for $q=2^e$ and $i\in \{1,2, \ldots,e{-}1\}$.
\smallskip

\begin{proposition}\label{prop:orders}
For $e\ge 2$ and $i\in \{1,2,\ldots,e-1\}$ let $(A,i)\in \PGam(2,2^e)$ and $r=\gcd(i,e)$. Then, the order of $(A,i)$ in $\PGam(2,2^e)$ is equal to $fe/r$, where either $f=2$, or $f$ is a divisor of $2^r\pm 1$.
\end{proposition}

{\bf Proof.}
Recalling our notation $[j]=2^j$ for powers of $2$ and also for the Frobenius automorphism $z\mapsto z^2$ applied $j$ times to field elements and matrix entries, for $m\ge 1$ one may check that a $t^{\rm th}$ power of the element $(A,i)\in \PGam(2,2^e)$ for $i\in \{1,2,\ldots,e-1\}$ has the form
\begin{equation}\label{eq:power}  (A,i)^t = (B,ti),\ \ \ {\rm where} \ \ \ B = AA^{[i]}A^{[2i]}\ldots A^{[(t-1)i]}\ . \end{equation}
Letting $t=e/r$ in \eqref{eq:power} and using $(e/r)i\equiv 0$ mod $e$, one obtains
\begin{equation}\label{eq:pow}  (A,i)^{e/r} = (B,0),\ \ \ {\rm where} \ \ \ B = AA^{[i]}A^{[2i]}\ldots A^{[(e/r-1)i]}\ . \end{equation}
An obvious consequence of \eqref{eq:pow} is that in $\PGam(2,2^e)$ the order of $(A,i)$ is an $(e/r)$-multiple of $\ord(B)$, the order of $B$ in $\PSL(2,2^e)$. In particular, if $B=I$ or of $B$ is an involution, the order of $(A,i)$ in $\PGam(2,2^e)$ is equal to $e/r$ and $2e/r$, respectively, proving our statement for $f\in \{1,2\}$.
\smallskip

To find out more, one may apply the field automorphism $z\mapsto z^{[i]}$ to $B$ in \eqref{eq:pow}. Since $A\in \PSL(2,2^e)$, one has $A^{[(e/r)i]} = A^{[e]}=A$, which gives $B^{[i]} = A^{[i]}A^{[2i]}\ldots A^{((e/r)-1)i]}A$, and hence $B^{[i]}= A^{-1}BA$. By our list of orders of elements of $\PSL(2,2^e)$, if $\ord(B) = m\ge 3$, then $m\mid 2^e\pm 1$ and $B$ is conjugate to  ${\rm dia}(\xi)$ for some primitive $m$-th root $\xi$ of unity, which is contained in $\GF(2^e)$ if $m\mid 2^e-1$, and in $\GF(2^{2e})\setminus \GF(2^e)$ if $m\mid 2^e+1$. From $B^{[i]}= A^{-1}BA$ it then follows that $B^{[i]}$ is conjugate to ${\rm dia}(\xi)$.
\smallskip

Since conjugation in $B^{[i]}=A^{-1}BA$ preserves traces, one has $\xi+\xi^{-1} = \xi^{[i]}+\xi^{-[i]}$, from which one obtains $\xi^{[i]}\in \{\xi,\, \xi^{-1}\}$. If $\xi^{[i]}=\xi$, then also $\xi^{[r]}=\xi$, so that $\xi$ lies in the subfield $\GF(2^r)$ of $\GF(2^e)$, and for the order $m$ of $\xi$ one has $m\mid 2^r-1$. But if $\xi^{[i]} =\xi^{-1}$ the situation is more complex, as we explain next.
\smallskip

If $m\mid 2^e-1$, then it is easy to show that $\xi^{[i]} =\xi^{-1}$ implies $\xi^{[r]}= \xi^{-1}$ with $|r|_2 < |e|_2$, where the symbol $|j|_2$ denotes the largest power of $2$ dividing $j$. It follows that $m\mid 2^r+1$ and $\xi$ is then contained in the subfield $\GF(2^{2r}) < \GF(2^e)$ but not in $\GF(2^r)$. In the opposite case, if $m\mid 2^e+1$, then, similarly, $\xi^{[i]} =\xi^{-1}$ implies that $\xi^{[r]}= \xi^{-1}$ with $|r|_2 = |e|_2$. Then one still has $m\mid 2^r+1$ and $\xi\in \GF(2^{2r}) \setminus \GF(2^r)$ but now $\GF(2^{2r})$ is not a subfield of $\GF(2^e)$, although obviously $\GF(2^{2r}) < \GF(2^{2e})$.
\smallskip

Summing up our findings, for every $i\in \{1,2,\ldots,e-1\}$ with $r=\gcd(i,e)$ and for any primitive $m$-th root of unity $\xi$ in $\GF(2^{2e})$ such that $m\ge 3$ and $m\mid 2^e\pm 1$ one has the following possibilities:
\smallskip

\noindent (a) if $m\mid 2^e-1$, then either $m\mid 2^r-1$, or $m\mid 2^r+1$ with $|r|_2 < |e|_2$\ ;
\smallskip

\noindent (b) if $m\mid 2^e+1$, then $m\mid 2^r+1$ with $|r|_2 = |e|_2$.
\smallskip

But in our setting, the order $m\ge 3$ of the element $\xi \in \GF(2^{2e})$ dividing $2^e\pm 1$ was equal to the order of both $B$ and ${\rm dia}(\xi)$ in $\PSL(2,2^e)$ or $\PSL(2,2^{2e})$. It follows that $m=\ord(B)$ is determined by one of (a), (b) above, and both can be summed up by stating that $m$ is a divisor of $2^r\pm 1$ distinct from $1$. In particular, one may check that $B^{[i]}=B^{[r]}$ in both cases. By \eqref{eq:pow} one then has $\ord(A,i)= \ord(A,r) = \ord(B){\cdot} (e/r)$, and this completes the proof. \hfill $\Box$
\medskip

It may be shown that all the orders listed in the statement of Proposition \ref{prop:orders} occur for every $e\ge 1$, but this is out of the scope of this article.
\smallskip


\section{The group ${\bf{P\Gamma L}}({\bf{2}},{\bf{2}}^{\bf{\textit p}})$ for an odd prime {\bf{\textit p}}}\label{sec:p}

\

To simplify the notation, in what follows we let $O=\off(1)$, and we also use $E$ to denote a conjugate of $\dia(\xi)$ for $\xi$ a primitive third root of unity, with the top left entry of $E$ being equal to $0$ and the remaining entries of $E$ being equal to $1$. We begin with an obvious consequence of Proposition \ref{prop:orders} in the situation when $e$ is an odd prime $p\ge 5$.

\begin{corollary}\label{cor:orders}
If $p\ge 5$ is an odd prime and $q=2^p$, then every element $\PGam(2,q)$ not contained in $\PSL(2,q)$ has order $p$, $2p$ or $3p$, and each of them occurs: $\ord(I,j)=p$, $\ord(O,j)=2p$ and $\ord(E,j)=3p$.
\end{corollary}
\smallskip

From now on we will assume that $q=2^p$ for some odd prime $p\ge 5$. Observe that since $p$ is not a divisor of the order $q(q^2{-}1)$ of $H=\PSL(2,q)$, from Proposition \ref{prop:orders} it also follows that every element of $G= \PGam(2,q)$ of order a multiple of $p$ lies outside $H$. We will continue with determining conjugacy classes in $G$ of elements of order divisible by $p$. This reduces to determining centralisers of such elements in the group $G$, and in our case it will turn out that it is sufficient to determine centralisers of $(I,j)$, $(O,j)$ and $(E,j)$ for $j\in \{1,2,\ldots,p-1\}$.
\smallskip

In general, the centraliser of an element $(A,j)\in G$ for $j\in \{1,2,\ldots,p{-}1\}$ consists of all elements $(U,k)\in G$, $k\in \{0,1,\ldots,p{-}1\}$, such that $(A,j)(U,k) = (U,k)(A,j)$, which is equivalent to the matrix equation
\begin{equation}\label{eq:conj}
  AU^{[j]} = UA^{[k]} \ \ {\rm for\ some} \ \ U= \left(\begin{matrix} a & b \\ c & d \\ \end{matrix}\right) \in G\ .
\end{equation}
Keeping the determinant requirement $ad+bc=1$ in mind, \eqref{eq:conj} reduces to four equations for the unknown entries $a,b,c,d\in \GF(2^e)$ of $U$. We will now go through the three possibilities for $A\in \{I,\,O,\, E\}$ which, as pairs $(A,j)$ for $j\in \{1,2,\ldots,p{-}1\}$, give elements of orders $p$, $2p$ and $3p$ in $G$, and determine their centralisers.
\smallskip

If $A=I$, the equation \eqref{eq:conj} reduces to $U^{[j]}=U$, and since the only proper subfield of $F=\GF(q)$ for $q=2^p$ is $\GF(2)$, it follows that $U$ is a non-singular $0$-$1$ matrix. There are $6$ such matrices, forming a group isomorphic to $S_3$, the symmetric group of degree $3$ and order $6$. Taking into the account the $p$ possible values of $k$ it follows that for every $j\in \{1,2,\ldots,p{-}1\}$, the centraliser in $G$ of  $(I,j)$ with $\ord(I,j)=p$ has order $6p$.
\smallskip

Taking now $A=O$, left multiplication by $A$ in \eqref{eq:conj} interchanges rows of $U^{[j]}$ while right multiplication by $A=A^{[k]}$ interchanges columns of $U$, giving equations $a^{[j]}=d$, $d^{[j]}=a$, $b^{[j]}=c$ and $c^{[j]}=b$. But then $z^{[2j]}=z$ for every $z\in \{a,b,c,d\}$, so that $U$ is again a $0$-$1$ matrix, this time with $a=d$ and $b=c$, which gives $U\in \{I,O\}$. Because of the range of values of $k$ it follows that the element $(O,j)$ of order $2p$ is centralised in $G$ by a subgroup of order $2p$.
\smallskip

For the third possibility when $A=E$, equations of \eqref{eq:conj} reduce to $c^{[j]}=b$, $b^{[j]} +d^{[j]}=c+d$, $a^{[j]} = b+d$, and $d^{[j]}=a+b$. Adding the first three gives $(a{+}b{+}c{+}d)^{[j]} = c$, adding the second to the fourth gives $a{+}b{+}c{+}d = b^{[j]}$, and adding the last two yields $(a{+}d)^{[j]} = a+d$. The two new equations for $a{+}b{+}c{+}d$ together with $c^{[j]}=b$ imply $c^{[3j]} = c$ and since $p$ is prime it follows that $c\in \{0,1\}$ and hence also $b=c\in \{0,1\}$. This reduces the second original equation to $(b{+}d)^{[j]} = b+d$, and as we have also derived that $(a{+}d)^{[j]} = a+d$, primality of $p$ implies that both $b+d$ and $a+d$ must be in $\{0,1\}$ and hence all of $a,b,c,d\in \{0,1\}$, with $b=c$. It now easily follows that $U$ is one of $I$, $E$, $E^2$, implying that for every $j\in \{1,2,\ldots,p{-}1\}$ the centraliser of the element $(E,j)$ in $G$ has order $3p$.
\smallskip

These facts constitute a large part of the proof of the following statement.

\begin{proposition}\label{prop:p-3p}
For $q=2^p$ with prime $p\ge 5$, let $G=\PGam(2,q)$ and $H=\PSL(2,q)$. For every $k\in \{1,2,3\}$, the group $G$ contains $p{-}1$ conjugacy classes of elements of order $kp$, all included in $G{\setminus}H$ and distinguished by their second coordinate $j\in \{1,2,\ldots, p{-}1\}$; each of these $p{-}1$ classes has size $|H|/6$ for $k=1$ and  $|H|/k$ for $k\in \{2,3\}$. Further, the group  $G$ contains exactly $|H|/6$, $|H|/2$ and $|H|/6$ cyclic subgroups of order $p$, $2p$ and $3p$, respectively, with each type of subgroups constituting a single conjugacy class in $G$.
\end{proposition}

{\bf Proof.} For each $j\in \{1,2,\ldots,p{-}1\}$ and for $k\in \{1,2,3\}$ let $C_{j,k}$ denote, respectively, the conjugacy class of the element $(I,j)\in G$, $(O,j))\in G$ and $(E,j)\in G$, of order $kp$. By the analysis preceding this proof in which sizes of centralisers of these three elements have been determined, one has $|C_{j,1}| = |G|/(6p) = |H|/6$, $|C_{j,2}|=|H|/2$ and $|C_{j,3}|=|H|/3$. For the sum of all elements in the total of $3(p{-}1)$ conjugacy classes $C_{j,k}$ for $j\in \{1,2,\ldots, p{-}1\}$ and $k\in \{1,2,3\}$ one then has
\begin{equation}\label{eq:sumC}
\sum_{j=1}^{p{-}1}\sum_{k=1}^3 |C_{j,k}| = (p{-}1)|H|(1/6 + 1/2 + 1/3) = |G{\setminus}H| \ .
\end{equation}
Since for $p\ge 5$ the group $H$ contains no element of order a multiple of $p$, it follows from \eqref{eq:sumC} that the group $G$ contains, for each $j\in \{1,2,\ldots,p{-}1\}$, a single conjugacy class (of size $|H|/6$) of elements of order $p$, all of which are conjugate to $(I,j)\in G$, a single conjugacy class (of size $|H|/2$) of elements of order $2p$, all conjugate to $(O,j)\in G$, and a single conjugacy class (of size $|H|/3$) of elements of order $3p$, all conjugate to $(E,j)\in G$.
\smallskip

Next, we determine the number of subgroups of order $p$, $2p$ and $3p$ and their conjugacy classes, beginning with subgroups of order $p$. The Sylow $p$-subgroups of $G$ are cyclic, of order $p$, and since the $p{-}1$ conjugacy classes of elements of order $p$ in $G$, each of size $|H|/6$, comprise all elements in $G$ of order $p$, it follows that $G$ contains exactly $|H|/6$ cyclic subgroups of order $p$, forming a single conjugacy class by Sylow theorems.
\smallskip

A cyclic group of order $2p$ contains $p{-}1$ elements of order $2p$, and since there are exactly $(p{-}1)|H|/2$ elements of order $2p$ in $G$, it follows that $G$ contains exactly $H/2$ cyclic subgroups of order $2p$. By calculations preceding the statement of Proposition \ref{prop:p-3p} one sees that for $j\in \{1,2,\ldots,p{-}1\}$ the centraliser of the {\em subgroup} $\langle (O,j)\rangle \cong C_{2p}$ coincides with the centraliser of the element $(O,j)$, and has order $2p$. This means that the conjugacy class of $\langle (O,j)\rangle$ in $G$ contains exactly $|G|/(2p) = |H|/2$ subgroups isomorphic to $C_{2p}$. But each of these contains $p{-}1$ elements of order $2p$ and hence the conjugacy class of $\langle (O,j)\rangle$ contains a total of $(p{-}1)|H|/2$ elements of order $2p$, which happens to be the total number of elements of order $2p$ in $G$. From \eqref{eq:sumC} and subsequent arguments it then follows that $G$ contains a single conjugacy class of cyclic subgroups of order $2p$.
\smallskip

Finally, note that a cyclic group of order $3p$ contains $2(p{-}1)$ elements of order $3p$. For $j\in \{1,2,\ldots, p{-}1\rangle$ we saw that the centraliser of $(E,j)$ in $G$ has order $3p$. This time, however, this does not extend to the group $\langle (E,j)\rangle \cong C_{3p}$, because one may check that for $\ell\in \{1,2\}$ and $p\equiv \ell$ mod $3$ one has $(E,j)^{\ell p+1} = (E^2,j)$ and since $E$ and $E^2$ are conjugate it follows that the centraliser of the {\em subgroup} $\langle (E,j)\rangle$ has order $6p$. Thus, the conjugacy class in $G$ of the subgroup $\langle (E,j)\rangle$ contains $|G|/(6p) = |H|/6$ subgroups isomorphic to $C_{3p}$. Now, each of these subgroups contains $2p{-}2$ elements of order $3$, giving $2(p{-}1)|H|/6 = (p{-}1)|H|/3$ distinct elements of order $3p$ in $G$. But the latter is precisely the total of all elements of order $3p$ in $G$ by arguments following \eqref{eq:sumC}, which implies that $G$ contains a single conjugacy class of cyclic subgroups of order $3p$. \hfill $\Box$
\medskip

It remains to determine maximal subgroups of $G=\PGam(2,q)$ for $q=2^p$, where $p$ is an odd prime and $p\ge 5$. By \cite[Table 7]{BDL}, the maximal subgroups of $H\cong \PSL(2,q) < G$ are the $q{+}1$ conjugate copies of the affine group $\Aff_q$ and the $q(q{\mp}1)/2$ mutually conjugate copies of the dihedral groups $D_{2(q{\pm}1)}$. It follows that the maximal subgroups of $G=\PGam(2,q)$ for $q=2^p$ are:
\begin{itemize}
\setlength{\itemsep}{0pt}
\setlength{\parskip}{0pt}
\setlength{\parsep}{0pt}
\item[{\rm (a)}] the unique copy of $H$ in $G$;
\item[{\rm (b)}] the $q{+}1$ mutually conjugate copies of the group $\Aff_q\rtimes C_p$ with elements $({\rm upp}(\xi^i,a),j)$ for $\xi$ and $j$ as in the case of $\Aff_q$ and for $j\in C_p$;
\item[{\rm (c)}] the groups $D_{2(q{-}1)}\rtimes C_p$ with elements $({\rm dia}(\xi^i),j)$ and $({\rm off}(\xi^i),j)$, with $\xi$ and $j$ as for $D_{2(q{-}1)}$ and with $j\in C_p$;
\item[{\rm (d)}] the groups $D_{2(q{+}1)}\rtimes C_p$, which, by Proposition \ref{prop:embed}, may be considered to consist of elements of the form $({\rm dia}(\xi^i),\veps(j))$ and $({\rm off}(\xi^i), \veps(j))$, where $\xi$ a primitive $(q{+}1)^{\rm st}$ root of unity in $\GF(q^2)$, $i\in \{0,1,2, \ldots,q\}$, $j\in C_p$, and $\veps$ is the mapping introduced in the proof of Proposition \ref{prop:embed}.
\end{itemize}
\noindent Observe that the sizes of the conjugacy classes of the latter three groups are the same in $G$ and $H$ because of the way the semi-direct product $G\cong HC_p$ is defined by the action of the Galois group of $\GF(q)$ on $H$.
\smallskip


\section{Determining the M\"obius function for ${\bf{P\Gamma L}}({\bf{2}},{\bf{2}}^{\bf{\textit p}})$}\label{sec:Mob}

\

From now on, in the interest of saving space we will denote the semidirect products $K\rtimes C_p$ for intersections of maximal subgroups $K$ of $G$ simply by $KC_p$. As alluded to in {\sl Step 1} of Section \ref{sec:strat}, to determine the values of the M\"obius function $\mu(K)$ on subgroups $K$ of $G$ it is sufficient to focus on those that are intersections of maximal subgroups of $G$. From the lists of maximal subgroups of $G$ and $H$ given in the previous sections it follows that intersection of maximal subgroups of $G$ has the form $\cap\{K\in {\cal S}\}$, where ${\cal S}$ is a subset of subgroups contained in the set $ \{H,\,\Aff_qC_p,\,\Aff_q,\,D_{2(q\pm 1)}C_p,\,D_{2(q\pm 1)}\}$. One may check that the cyclic group $C_{q-1}$ may be obtained as $\Aff_q\cap D_{2(q-1)}$, the group $C_2$ may be obtained as an intersection $D_{2(q-1)}\cap D_{2(q+1)}$ of suitable copies of the two dihedral groups, and the trivial group $C_1$ arises as $\Aff_q\cap D_{2(q+1)} = C_{q-1}\cap C_2$. This implies that the list of all subgroups $K$ of $G$ that arise as intersections of its maximal subgroups is as in Table \ref{tab:max}; for each of them there is just one conjugacy class in $G$ and its size ${\rm conj}(K)$ is given in the second line of the table.
\smallskip

\begin{table}[ht]
\scriptsize	
\centering
\begin{tabular}{|c||c|c|c|c|c|c|c|c|c|c|c|}
\hline \xrowht[()]{8pt}
$K< G$ & $H$ & $\Aff_qC_p$ & $D_{2(q{\pm}1)}C_p$ & $\Aff_q$ & $D_{2(q{\pm}1)}$ & $C_{q{-}1}C_p$ & $C_{q{-}1}$ & $C_{2p}$ & $C_2$ & $C_p$ & $C_1$ \\ \hline \xrowht[()]{6pt}
${\rm conj}(K)$ & $1$ & $q{+}1$ & $q(q{\mp}1)/2$ & $q{+}1$ & $q(q{\mp}1)/2$ & $q(q{+}1)/2$ & $q(q{+}1)/2$ & $q(q^2{-}1)/2$ & $q^2{-}1$ & $q(q^2{-}1)/6$ & $1$   \\ \hline
\end{tabular}
\caption{Intersections $K$ of maximal subgroups of $G$ and their conjugacy class sizes.}\label{tab:max}
\end{table}

For a subgroup $J$ of $G$ let $\is(J)$ denote the number of isomorphic copies of $J$ in $G$, and for a pair of subgroups $K$ and $L$ of $G$ we let $\is(L{>}\lfix K)$ and $\is(K{<}\lfix L)$ denote, respectively, the number of isomorphic copies of $L$ in $G$ properly containing a fixed subgroup $K<G$, and the number of isomorphic copies of $K$ being contained in a fixed group $L<G$. Of particular importance for determination of the M\"obius function of the lattice of subgroups of $G$ will be the quantities $\is(L{>}\lfix K)$. To see why, for a subgroup $K$ of $G$ let ${\cal O}(K)$ denote a set of representatives of all isomorphism classes of subgroups of $G$ properly containing $K$. The recursion of \eqref{eq:Mob} is then equivalent to the formula
\begin{equation}\label{eq:mu(K)}
\mu(K) = -\sum_{L\in {\cal O}(K)} \mu(L)\cdot\is(L{>}\lfix K) \ ;
\end{equation}
note that correctness of the formula \eqref{eq:mu(K)} follows from the fact that all the subgroups isomorphic to a particular maximal subgroup of $G$ form a single conjugacy class in $G$.
\smallskip

The numbers $\is(L{>}\lfix K)$ can be determined by doubly counting the number of pairs $(K^\ast,L^\ast)$ of subgroups of $G$ such that $K^\ast\cong K$, $L^\ast\cong L$ and $K^\ast{<}L^\ast$. Indeed, the number of such pairs can be obtained either by counting all subgroups of $G$ isomorphic to $K$ and then determine in how many copies of $L$ a particular specimen of $K$ is present, or by first counting the number of copies of $L$ in $G$ and then determine how many copies of $K$ are contained in a particular specimen of $L$; this all relies on the fact that subgroups of $G$ isomorphic to a given maximal subgroup $K<G$ form a single conjugacy class in $G$. In our notation, this double counting is encapsulated by the equation
\begin{equation}\label{eq:double}
\is(K)\cdot \is(L{>}\lfix K) = \is(L)\cdot \is(K{<}\lfix L)\ .
\end{equation}

We will use \eqref{eq:double} to determine the values of $\is(L{>}\lfix K)$ in \eqref{eq:mu(K)} with help of the other three values appearing in \eqref{eq:double}, which are either known or easy to find. For example, using knowledge about maximal subgroups of $H$ one may verify that $\is(\Aff_q) = \is(\Aff_qC_p) = q{+}1$, $\is(D_{2(q{\pm}1)}) = \is(D_{2(q{\pm}1)}C_p) = q(q{\mp}1) /2$, $\is(C_{q{-}1}) = \is(C_{q{-}1}C_p) = q(q{+}1)/2$, and by Proposition \ref{prop:p-3p} one has $\is(C_{2p}) = q(q^2{-}1)/2$ and $\is(C_{p}) = q(q^2{-}1)/6$. With this preparation we are ready to determine the M\"obius function for $G$.

\begin{theorem}\label{thm:Mob}
Let $q=2^p$ for prime $p\ge 5$ and let $G=\PGam(2,q)$. The non-zero values of the M\"obius function of $G$ are given by the following table:
\begin{table}[ht]
\footnotesize	
\centering
\begin{tabular}{|c||c|c|c|c|c|c|c|c|c|c|c|}
\hline \xrowht[()]{8pt}
$K$ & $H$ & $\Aff_qC_p$ & $D_{2(q{\pm}1)}C_p$ & $\Aff_q$ & $D_{2(q{\pm}1)}$ & $C_{q{-}1}C_p$ & $C_{q{-}1}$ & $C_{2p}$ & $C_2$ & $C_p$ & $C_1$ \\ \hline \xrowht[()]{6pt}
$\mu_K$ & $-1$ & $-1$ & $-1$ & $1$ & $1$ & $2$ & $-2$ & $2$ & $-q$ & $-6$ & $q(q^2{-}1)$   \\ \hline
\end{tabular}
\caption{Values of the M\"obius function at subgroups $K$ from Table \ref{tab:max}.}\label{tab:mu}
\end{table}
\end{theorem}

{\bf Proof.} To begin with, observe that by \eqref{eq:Mob} the values of the M\"obius function $\mu(K)$ for intersections of maximal subgroups $K$ of $G$ from Table \ref{tab:max} are obvious for maximal subgroups of $G$, that is, for $K\in \{H,\,\Aff_qC_p,\, D_{2(q{\pm}1)}C_p\}$, with $\mu(K)=-1$, and for maximal subgroups $K\in \{\Aff_q,\,D_{2(q{\pm}1)}\}$ of $H$ (for which $\mu(K)=1$).
\smallskip

To evaluate $\mu(K)$ for $K=C_{q{-}1}C_p$ by \eqref{eq:mu(K)} one takes ${\cal O}(K) = \{\Aff_qC_p,\, D_{2(q{-}1)} C_p,\,G\}$ with the obvious values $\is(L{>}\lfix K)=1$ for $L\in \{D_{2(q{-}1)} C_p,\,G\}$. The structure of $\Aff_q\cong {\rm GF}(q)\rtimes C_{q-1}$ implies that $\Aff_q$ contains exactly $q = |{\rm GF}(q)|$ distinct copies of $C_{q-1}$, each generated by the element $(\upp(a,1)$ for $a\in {\rm GF}(q)$, so that $\is(K{<} \lfix \Aff_qC_p) =q$. Recalling that $\is(K)=\is(C_{q{-}1})= q(q{{+}}1)/2$ and $\is(\Aff_qC_p) = \is(\Aff_q)=q{+}1$, from \eqref{eq:double} one then has $\is(\Aff_q C_p {>} \lfix K)=2$. Substituting the obtained values of $\is(L{>}\lfix K)$ for $L\in {\cal O}(K)$ in the equation \eqref{eq:mu(K)} eventually gives $\mu(C_{q{-}1}C_p)=2$.
\smallskip

Continuing with evaluation of $\mu(K)$ for $K = C_2C_p\cong C_{2p}$, one may take ${\cal O}(K) = \{ \Aff_qC_p,\,
D_{2(q{\pm}1)}C_p,\,G\}$. Here, the two obvious values entering \eqref{eq:mu(K)} are $\is(L{>}\lfix K) =1$ for $L\in \{G,\,\Aff_qC_p\}$. To determine $\is(D_{2(q{\pm}1)}C_p {>}\lfix K)$ by means of \eqref{eq:double} one uses the known values of $\is(D_{2(q{\pm}1)}C_p) = \is(D_{2(q{\pm}1)})= q(q{{\mp}}1)/2$ together with $\is(K) = q(q^2{-}1)/2$. To proceed, let us represent $D_{2(q{\pm}1)}$ as the group of generated by the elements ${\rm dia}(x)$ and ${\rm off}(1)$ for $x$ a primitive $(q{\pm}1)^{\rm st}$ root of unity; this may be done without loss of generality even when $x$ is a primitive $(q{+}1)^{\rm st}$ root of unity using the well-known embedding $\PSL(2,q)\to \PSL(2,q^2)$ and the embedding $\PGam(2,q)\to \PGam(2,q^2)$ provided by Proposition \ref{prop:embed}. It may then be checked that the element $({\rm off}(x^i),j)\in G$ generates a copy of $C_{2p}$ for every $i\in \{1,2,\ldots,q{-}1\}$ and $j\in\{1,2,\ldots,p{-}1\}$. It then easily follows that $D_{2(q{\pm}1)}C_p$ contains exactly $q{\pm}1$ distinct copies of $C_{2p}$, so that $\is (K{<}\lfix D_{2(q{\pm}1)}C_p) = q{\pm}1$ for $K=C_{2p}$. Substituting these values in \eqref{eq:double} yields $\is(D_{2(q{\pm}1)}C_p {>} \lfix K) = 1$ for $K=C_{2p}$, and \eqref{eq:mu(K)} then gives $\mu(C_{2p})=2$.
\smallskip

For $K=C_2$ one may work with ${\cal O}(K) = \{C_{2p},\,D_{2(q{\pm}1)},\,\Aff_q,\, D_{2(q{\pm}1)}C_p,\, \Aff_qC_p,\, H,\, G\}$. Since $\is(C_2) =q^2{-}1$, $\is(C_{2p})= q(q^2{-}1)/2$ and $\is(C_2{<}\lfix C_{2p})= 1$, by \eqref{eq:double} one has $\is(C_{2p} {>}\lfix C_2) = q/2$. A similar calculation for $K=C_2$ and $L\in \{D_{2(q{\pm}1)},\,D_{2(q{\pm}1)}C_p\}$, with $\is(L)= q(q{\mp}1)/2$ and $\is(C_2{<}\lfix L) = q{\pm}1$ in \eqref{eq:double}, gives $\is(L {>}\lfix C_2) = q/2$. But the situation is different for $L\in \{\Aff_q, \,\Aff_qC_p\}$, with $\is(L)= q{+}1$ and $\is(C_2{<}\lfix L) = q{-}1$, giving this time $\is(L {>} \lfix C_2) = 1$ from \eqref{eq:double}. A calculation using \eqref{eq:mu(K)} with the values of $\is(L {>}\lfix C_2)$ just determined for every $L\in {\cal O}(C_2)$ and with the previously obtained values of $\mu$ then leads to the value of $\mu(C_2)=-q$.
\smallskip

To find $\mu(C_p)$ one takes ${\cal O} = \{ C_{2p},\,C_{q-1}C_p,\, D_{2(q{\pm}1)}C_p,\,\Aff_qC_p,\, G\}$. Since $\is(C_p) = q(q^2{-}1)/6$ and $\is(C_{2p}) = q(q^2{-}1)/2$ by Proposition \ref{prop:p-3p}, double counting using the obvious fact that $\is(C_p{<}\lfix C_{2p}) = 1$ yields $\is(C_{2p} {>} \lfix C_p) = 3$. Combining the last equation with $\is(\Aff_qC_p {>} \lfix C_{2p}) = 1$ established earlier, one concludes that $\is (\Aff_qC_p {>} \lfix C_{p}) = 3$. (For example, for the fixed subgroup $J=\langle (I,1)\rangle \cong C_p$ the three copies of $\Aff_qC_p$ containing $J$ are those obtained as conjugates of the standard copy of $\Aff_q$ by elements of the group $\langle E\rangle\cong C_3$.) Further, for $L\in \{C_{q-1}C_p,\, D_{2(q{-}1)}C_p \}$ it is easy to see that $\is(C_p{<}\lfix L) = q{-}1$ and then \eqref{eq:double} with the values of $\is(C_p) = q(q^2{-}1)/6$ and $\is(D_{2(q{-}1)}C_p)= q(q{+}1)/2$ one obtains $\is(L {>}\lfix C_p)= 3$. \smallskip

The situation for $L_0= D_{2(q{+}1)}C_p$ is exceptional. Our calculation of the centraliser of $(I,j)$ in the arguments preceding the statement of Proposition \ref{prop:p-3p} and also in the proof of this Proposition imply that the {\em group} $\langle(I,j)\rangle$, and hence every subgroup of $G$ isomorphic to $C_p$, has centraliser of size $6p$, isomorphic to $\langle O,E\rangle C_p$. Since the latter is a subgroup of $L_0$, from the aforementioned calculations it also follows that the centraliser of $C_p$ {\em in the subgroup} $L_0$ has order $6p$ as well, so that $\is(C_p{<}\lfix L_0) = |L_0|/(6p) = 2(q{+}1)p/(6p) = (q{+}1)/3$. This together with $\is(L_0) = q(q{-}1)/2$ and $\is(C_p) = q(q^2-1)/6$ yields $\is(L_0 {>}\lfix C_p)  = 1$, and substituting all the obtained values in \eqref{eq:mu(K)} results in $\mu(C_p) = -6$.
\smallskip

Finally, the value of $\mu(C_1)$ is simply determined from all the previously obtained values of $\mu$ by means of \eqref{eq:mu(K)} and the observation that for every subgroup $L{<}G$ the number $\is(L{>}\lfix C_1)$ is equal to the number of isomorphic (equivalently, conjugate) copies of $L$ in $G$, giving $\mu(C_1) = q(q^2{-}1)$. \hfill $\Box$
\smallskip


\section{Enumeration of orientably-regular maps \\ of unspecified type on the group ${\bf{P\Gamma L}}({\bf{2}},{\bf{2}}^{\bf{\textit p}})$}\label{sec:enum}

\

By {\sl Step 3} of Section \ref{sec:strat}, where we outlined the process of enumeration of isomorphism classes of orientably-regular maps $M$ of unspecified type with $\Aut(M) \cong G=\PGam(2,2^p)$ for prime $p\ge 5$, the number ${\rm map}(G)$ of such classes is given by
\begin{equation}\label{eq:Map(G)} {\rm map}(G) = \ep(G)/|G| \ .
\end{equation}
It thus remains to calculate the number $\ep(G)$ of epimorphisms $\Delta(2, \infty, \infty) \to G$  appearing in \eqref{eq:Map(G)}. We recall that, by {\sl Step 2} of Section \ref{sec:strat}, the quantity $\ep(G)$ is, at the same time, equal to the number of generating pairs $(x,y)$ of $G$ such that $(xy)^2=1$, and is given by means of values $\mu(K)$ of the M\"obius function on subgroups $K<G$ by the formula
\begin{equation}\label{eq:Epi(G)}
\ep(G) = \sum_{K\le G}\mu(K)\cdot \ho(K), \ \ {\rm where} \ \ \ho(K) = |\{z\in K\,;\,z^2=1\}|\cdot |K|\ .
\end{equation}
If ${\cal O}$ is the set of subgroups of $G$ of the form $LC_p$ for $L\in \{H,\,\Aff_q,\,D_{2(q{\pm}1)},\, C_{q{-}1},\,C_2,\,C_1\}$, that is, the set of intersections of maximal subgroups of $G$ including $G$ itself, then the formula for $\ep(G)$ in \eqref{eq:Epi(G)} may further be written in the form
\begin{equation}\label{eq:EG}
\ep(G) = \sum_{K\in {\cal O}}\mu(K)\cdot \is(K)\cdot\ho(K) \ .
\end{equation}
To calculate the values of $\ho(K)$ in \eqref{eq:EG} it suffices to determine the number of involutions {\em together} with the identity element in the groups $K\in {\cal O}$. This gives $\ho(K)=|K|$ if $|K|$ is odd, and  $\ho(LC_p)=\ho(L)\cdot p$ for $L$ as above, with $\ho(H)=q^2|H|$, $\ho(\Aff_q) = q|\Aff_q|$, $\ho(D_{2(q{+}1)}) =2(q{+}2)(q{+}1)$, $\ho(D_{2(q{-}1)})=2q(q{-}1)$ and $\ho(C_2)=4$. The remaining parameters in \eqref{eq:EG} depending on subgroups $K\in {\cal O}$ have been determined in Section \ref{sec:Mob}. For clarity, the subgroups $K<G$ and the corresponding values of $\ho(K)$, $\is(K)$ and $\mu(K)$ entering the formula in \eqref{eq:EG} are listed in Table \ref{tab:X}.
\begin{table}[ht]
\footnotesize
\centering
\begin{tabular}{|c||c|c|c|c|c|c|}
\hline \xrowht[()]{8pt}
$K$      & $G$       & $H$      & $\Aff_qC_p$   & $\Aff_q$     & $D_{2(q{+}1)}C_p$  & $D_{2(q{+}1)}$  \\ \hline \xrowht[()]{6pt}
$\ho(K)$ & $q^2p|H|$ & $q^2|H|$ & $q^2(q{-}1)p$ & $q^2(q{-}1)$ & $2(q{+}2)(q{+}1)p$ & $2(q{+}2)(q{+}1)$   \\ \hline \xrowht[()]{6pt}
$\is(K)$ & $1$       & $1$      & $q{+}1$       & $q{+}1$      & $q(q{-}1)/2$       & $q(q{-}1)/2$  \\ \hline \xrowht[()]{6pt}
$\mu(K)$ & $1$       & $-1$     & $-1$          & $1$          & $-1$               & $1$       \\ \hline
\end{tabular}
\smallskip

\begin{tabular}{|c||c|c|c|c|c|c|c|c|}
\hline \xrowht[()]{8pt}
$K$      & $D_{2(q{-}1)}C_p$ & $D_{2(q{-}1)}$ & $C_{q{-}1}C_p$ & $C_{q{-}1}$  & $C_{2p}$       & $C_2$     & $C_p$          & $C_1$         \\ \hline \xrowht[()]{6pt}
$\ho(K)$ & $2q(q{-}1)p$      & $2q(q{-}1)$    & $(q{-}1)p$     & $q{-}1$      & $2\cdot 2p$    & $4$       &   $p$            & $1$           \\ \hline \xrowht[()]{6pt}
$\is(K)$ & $q(q{+}1)/2$      & $q(q{+}1)/2$   & $q(q{+}1)/2$   & $q(q{+}1)/2$ & $q(q^2{-}1)/2$ & $q^2{-}1$ & $q(q^2{-}1)/6$ & $1$           \\ \hline \xrowht[()]{6pt}
$\mu_K$  & $-1$              & $1$            & $2$            & $-2$         & $2$            & $-q$      &  $-6$           & $q(q^2{-}1)$  \\ \hline
\end{tabular}
\caption{Values of $\ho(K)$, $\is(K)$ and $\mu(K)$ for intersections of maximal subgroups $K{<}G$.}\label{tab:X}
\end{table}

As the next step, let us group the terms appearing in \eqref{eq:EG} into pairs of the form $\{LC_p,\,L\}$, where $L$ ranges over the set ${\cal L}=\{H,\,\Aff_q,\, D_{2(q{+}1)},\, D_{2(q{-}1)},\, C_{q{-}1},\,C_2,\,C_1\}$. Then, one may check that every partial sum
\begin{equation}\label{eq:L}
\mu(LC_p)\,\is(LC_p)\,\ho(LC_p) + \mu(L)\,\is(L)\,\ho(L), \ \ L\in {\cal L}
\end{equation}
appearing in \eqref{eq:EG} and consisting of products of three values in neighbouring columns headed $LC_p$ and $L$ for $L\in {\cal L}$, turns out to be a multiple of $(p-1)|H|$. This greatly facilitates evaluation of $\ep(G)$ in \eqref{eq:EG}, and after simplification one obtains the result in the form $\ep(G) = (p-1)|H| (q-1)(q-2)$.
\smallskip

Finally, following {\sl Step 3} of the strategy outline of Section \ref{sec:strat}, the sought number ${\rm map}(G)$ of orientably-regular maps $M$ of unspecified type and such that $\Aut(M)\cong G$ is equal to ${\rm map}(G) = \ep(G)/|G| = \ep(G)/(p|H|) = (p-1)(q-1)(q-2)/p$. Observe that the latter is a whole number, since $p$ is a divisor of $q-2 = 2(2^{p{-}1}-1)$ by Fermat's Little Theorem. Summing up, we have proved:
\smallskip

\begin{theorem}\label{thm:unspec}
Let $q=2^p$ for an arbitrary prime $p\ge 5$. Then, the number ${\rm map}(G)$ of pairwise non-isomorphic orientably-regular maps of unspecified type and with automorphism group isomorphic to $G=\PGam(2,q)$, is equal to $(p-1)(q-1)(q-2)/p$. 
\end{theorem}
\smallskip

Comparing the result of Theorem \ref{thm:unspec} with that of \cite[p. 302]{DJ} it follows that the number of isomorphism classes of orientably-regular maps $M$ such that $\Aut(M)\cong \PGam(2,q)$ for $q=2^p$ with prime $p\ge 5$ is equal to $(p{-}1)$-times the number of such classes of maps with $\Aut(M)\cong \PSL(2,q)$.
\smallskip


\section{\bf Character table for ${\bf{P\Gamma L}}({\bf{2}},{\bf{2}}^{\bf{\textit p}})$}\label{sec:char}

\

The character table for the group $H=\PSL(2,q)$ for $q$ a power of $2$ is well known, and we reproduce it here in a form almost identical with the one given in \cite[Theorem 38.2]{Dor}), in order to explain how it extends to a character table of $G=\PGam(2,q)$ when restricted to $q=2^p$ for prime $p\ge 5$. To save space, from this point on we will use the subscript notation $A_t$ for elements $(A,t)\in G$ with $A\in H$ and $t\in C_p$. The  conjugacy class containing a particular element, say, $A\in H$ or $A_t\in G$, will be denoted by $[A]$ and, respectively, $[A_t]$; the group will always be clear from the context. Also, it will prove useful to work with a somewhat non-standard notation and let, for an element $\xi$ of a field, the symbol $\xi(d)$ denote the quantity $\xi^{\m d} + \xi^{-d}$ for integer $d$. With this we also extend the definition of the element $E\in H$ from Section \ref{sec:p} by letting $E(\xi(d))$ denote the $2\times 2$ matrix with top left entry equal to zero, off-diagonal entries equal to $1$, and the bottom right entry equal to $\xi(d) = \xi^{\m d} + \xi^{-d}$.
\smallskip

We will now briefly revisit conjugacy classes and the table of irreducible characters of the group  $H=\PSL(2,q)$.
\smallskip

{\bf Conjugacy classes of $\bf{\textit H}$.} Let $\zeta$ be a primitive element of $\GF(q)$ and let $\eta$ be a primitive $(q{+}1)^{\rm st}$ root of unity in $\GF(q^2)$, so that, in the new notation, $\eta(1) = \eta{+}\eta^{-1}\in \GF(q)$. The group $H$ contains $(q{-}2)/2$ conjugacy classes of the form $[E(\zeta(j))]$ for $j\in \{1,2, \ldots,(q{-}2)/2\}$, and $q/2$ conjugacy classes of the form $[E(\eta(k))]$ for $k\in \{1,2,\ldots,q/2\}$, with respective centralisers of order $q-1$ and $q+1$. Together with the trivial class $[I]$ of the unit element $I$, and a single conjugacy class $[E(0)] = [\off(1)]$ of involutions with centraliser of order $q$, these comprise a total of $q+1$ conjugacy classes of $H$.
\smallskip

{\bf Irreducible characters of $\bf{\textit H}$.} The symbols $\iota$ and ${\rm St}$ denote the trivial and the Steinberg character. For the remaining character we use the symbols $\chi$ and $\theta$ as in \cite[Theorem 38.2]{Dor}), indexed, respectively, by $\ell\in \{1,2,\ldots, (q{-}2)/2\}$ and $m\in \{1,2,\ldots, q/2\}$. With $\alpha = \exp(2\pi i/(q{-}1))$ and $\beta = \exp(2\pi i/(q{+}1))$ being primitive complex $(q{-}1)^{\rm st}$ and $(q{+}1)^{\rm st}$ roots of unity, the corresponding values of $\chi_\ell$ and $\theta_m$ on the conjugacy classes $[E(\zeta(j))]$ and $[E(\eta(k))]$ are, respectively, $\alpha(j\ell) = \alpha^{\m j\ell}{+} \alpha^{-j\ell}$ and $-\beta(km) = -(\beta^{\m km} {+} \beta^{-km})$, giving a total of $q{+}1$ irreducible characters of $H$ in Table \ref{tab:PSL}.
\smallskip

\begin{table}[ht]
\small
	\centering
\begin{tabular}{|c||c|c|c|c|}
\hline \xrowht[()]{8pt}
Conjugacy class & $[I]$ & $[E(0)]$ & $[E(\zeta(j))]$ & $[E(\eta(k))]$ \\ \xrowht[()]{3pt}
Number of classes & $1$ & $1$ & $(q{-}2)/2$ & \ \ $q/2$ \ \ \\ \hline\hline \xrowht[()]{8pt}
$\iota$ & $1$ & $1$ & $1$ & $1$ \\ \hline \xrowht[()]{6pt}
${\rm St}$ & $q$ & $0$ & $1$ & $-1$ \\ \hline \xrowht[()]{6pt}
$\chi_\ell$ & $q{+}1$ & $1$ & $\alpha(j\ell)$ & $0$  \\ \hline \xrowht[()]{6pt}
$\theta_m$ & $q{-}1$ & $-1$ & $0$ & $-\beta(km)$ \\ \hline
\end{tabular}
\caption{Table of irreducible characters of $H=\PSL(2,q)$ for $q$ a power of $2$.}\label{tab:PSL}
\end{table}
\smallskip

To set up a character table of the group $G=\PGam(2,q)$, where $q=2^p$ for prime $p\ge 5$, we need to introduce more notation and facts. Keeping $\zeta$ and $\eta$ as before and sticking to the notational conventions introduced earlier, for $\zeta(j) = \zeta^{\m j}{+}\zeta^{-j}$ and $\eta(k) = \eta^{\m k}{+} \eta^{-k}$ let us introduce the sets $R_\zeta = \{\zeta(j) \mid 1\le j\le (q{-}2)/2 \}$ and $R_\eta = \{\eta(k) \mid 1\le k\le q/2 \}$. Let $\phi:\ z\mapsto z^2$ be the generator of the cyclic automorphism group of $\GF(q)$, of order $p$.
\smallskip

Determining the size of an orbit of $\langle\phi\rangle$ on the sets $R_\zeta$ and $R_\eta$ is equivalent to finding, for $j\in \{1,2,\ldots, (q{-}2)/2\}$ and $k\in \{1,2,\ldots, q/2\}$, the smallest positive $s,t \le p$ such that $\zeta(2^sj) = \zeta(j)$ and $\eta(2^tk) = \eta(k)$, respectively. These conditions are equivalent, respectively, to $2^p-1$ dividing $(2^s\pm 1)j$ and $2^p+1$ dividing $(2^t\pm 1)k$. But $\gcd(2^p-1,2^s\pm 1) = 1$ for $s<p$, and $\gcd(2^p+1,2^t\pm 1)=1$ for $t < p$ except for the single case when $t=1$ and $\gcd(2^p+1, 2^t+1) =3$. It follows that, for $j\le (q{-}2)/2$, the only value of $s$ in the given range that satisfies the divisibility condition $(2^p-1) \mid (2^s\pm 1)j$ is $s=p$, with $(q{-}2)/(2p) = (2^{p{-}1}{-}1)/p$ orbits of $\langle \phi \rangle$ of size $p$ on the set $R_\zeta$. Similarly, for $k\le q/2$  the only possibilities for $t$ in the given range to satisfy the condition $(2^p+1) \mid (2^t\pm 1)j$ are $t\in \{1,p\}$, with a singleton orbit of $\langle\phi \rangle$ on $R_\eta$ for $t=1$ and $k= (q{+}1)/3$, and the remaining $(q/2{-}1)/p=(2^{p-1}{-}1)/p$ orbits on $R_\eta$ of size $p$ for $t=p$.
\smallskip

For each orbit of the action of the group $\langle\phi\rangle$ on $R_\zeta$ and $R_\eta$ we choose a representative $\zeta(j_\phi)$ and $\eta(k_\phi)$, respectively; we may assume that $1\le j_\phi \le (q{-}2)/2$ and $1\le k_\phi \le q/2$, so that for the singleton orbit $\{\eta(\kappa)\}$ for $\kappa = (q{+}1)/3$ we have $k_\phi = \kappa$. Let ${\rm Rep}(R_\zeta)$ and ${\rm Rep}(R_\eta)$ be respective sets of such representatives, and let ${\cal P}_\zeta = \{j_\phi \mid \zeta(j_\phi)\in {\rm Rep}(R_\zeta)\}$ and ${\cal P}_\eta = \{k_\phi \mid k_\phi\ne\kappa \ {\rm and}\ \eta(k_\phi)\in {\rm Rep}(R_\eta)\}$ be the corresponding sets of powers applied to $\zeta$ and $\eta$ in the sets ${\rm Rep}(R_\zeta)$ and ${\rm Rep}(R_\eta)$, except for $\kappa$ in the latter set. Arguments from the previous paragraph imply that $|{\rm Rep}(R_\zeta)| = |{\rm Rep}(R_\eta)| = (q-2)/(2p)$.
\smallskip

For the action of $\langle\phi\rangle$ as a subgroup of the automorphism group of $H=\PSL(2,q)$ on its conjugacy classes, our analysis of orbits of $\langle\phi\rangle$ on the sets $R_\zeta$ and $R_\eta$ implies the following. Omitting from now on the subscript $\phi$ at representative exponents in ${\cal P}_\zeta$ and ${\cal P}_\eta$, for every $j\in {\cal P}_\zeta$ the $p$ conjugacy classes of $H$ represented by elements $E(\zeta(2^uj))$ for $u \in \{0,1,\ldots p{-}1\}$ give rise (under the action of $\langle\phi\rangle$) to a single conjugacy class of $G$ with representative element $E(\zeta(j))_0$, totalling $(q-2)/(2p)$ such conjugacy classes of $G$. The situation is similar for $R_\eta$, where for every $k\in {\cal P}_\eta$ the $p$ conjugacy classes of $H$ represented by elements $E(\eta(2^u k))$ for $u\in \{0,1,\ldots,p{-}1\}$ of $H$ form (again under the action of $\langle\phi \rangle$) a single conjugacy class of $G$ with representative element $E(\eta(k))_0$. Together with the conjugacy class with representative element $E(\eta(\kappa))_0$, one has in total $1+(q-2)/(2p)$ conjugacy classes of $G$ of this kind. For simplification, note that $\eta^{\pm\kappa}$ are primitive $3^{\rm rd}$ roots of unity, so that $\eta(\kappa) = \eta^{\m \kappa} + \eta^{-\kappa} = 1$, and hence $E(\eta(\kappa))_0 = E(1)_0$. Observe also that $\langle\phi\rangle$ induces orbits of size $p$ on the subsets of elements of the sets ${\cal P}_\zeta$ and ${\cal P}_\eta$ divisible by $3$, so that each of the two sets contains exactly $(q-2)/(6p)$ multiples of $3$.
\smallskip

With this preparation we are ready to describe entries of a table of irreducible characters for the group $G=\PGam(2,q)$ with $q=2^p$ for prime $p\ge 5$.
\smallskip

{\bf Conjugacy classes of $\bf{\textit G}$.} Beside the trivial class $[I_0]$ of the unit element, the group $G$ contains a single conjugacy class $[E(0)_0]$ of involutions, with centraliser of order $pq$, and also a single conjugacy class $[E(1)_0]$ of elements of order $3$, with centraliser of order $p(q+1)$; note that $E(1)= E(\eta(\kappa))$ because both $\eta^\kappa$ and $\eta^{-\kappa}$ are primitive $3^{\rm rd}$ roots of $1$ in $\GF(q)$ and so $\eta^\kappa + \eta^{-\kappa}=1$. Further, $G$ contains $(q-2)/(2p)$ conjugacy classes of the form $[E(\zeta_j)_0]$ for $j\in {\cal P}_\zeta$, with centraliser of order $q-1$, and the same number of conjugacy classes of the form $[E(\eta_k)_0]$ for $k\in {\cal P}_\eta$ (recall that here $k\ne\kappa$), with centraliser of order $q+1$. Elements of the form $A_t\in G$ for $t\in \{1,2,\ldots,p{-}1\}$ form exactly $3(p-1)$ conjugacy classes, represented by the elements $I_t$, $E(0)_t$ and $E(1)_t$ of orders $p$, $2p$ and $3p$, respectively, with orders of centralisers $6p$ for $I_t$, $2p$ for $E(0)_t$, and $3p$ for $E(1)_t$, $1\le t\le p{-}1$. This gives a total of $3 + 2(q-2)/(2p) + 3(p-1) = 3p+(q-2)/p$ conjugacy classes of $G$.
\smallskip

{\bf Irreducible characters of $\bf{\textit G}$.} In our tables we will use the same symbols for characters of $H$ and $G$, and to avoid clashes we will temporarily use the lower prescript $`H'$ for characters of $H$ that serve as basis for lifting or inducing characters from $H$ to $G$.
\smallskip

Let $\gamma = \exp(2\pi i/p)$ be a complex $p^{\rm th}$ root of unity. In Table \ref{tab:PGam} for $G$, the values of the characters $\iota_s$ for $s\in \{0,1,\ldots,p{-}1\}$ on the conjugacy classes $I_t$, $E(0)_t$ and $E(1)_t$ for $t\in \{1,2,\ldots,p{-}1\}$ are equal to $\gamma^{st}$. Observe that the linear characters $\iota_s$ are lifts of the $p$ linear characters of the quotient group  $G/H$ (isomorphic to the cyclic group $C_p$ of order $p$) to the entire group $G=\PGam(2,q)$, with $\iota_0$ being the trivial character. The characters ${\rm St}{\cdot}\iota_s$ for $s\in \{0,1,\ldots, p{-}1\}$ are products of the Steinberg character ${\rm St}$ of $G$, of dimension $q$, with the $p$ linear characters. The $(q-2)/(2p)$ characters $\chi_\ell$ of $G$ for $\ell \in {\cal P}_\zeta$, of dimension $p(q+1)$, are obtained as those induced by any of the characters $\leftindex_{H} \chi_{\ell{\m '}}$ of the subgroup $H<G$ from Table \ref{tab:PSL} for $\ell{\m '}=2^u\ell$, $0\le u\le p{-}1$. Similarly, the same number $(q-2)/(2p)$ of characters $\theta_m$ of $G$ for $m \in {\cal P}_\eta$, of dimension $p(q-1)$, are obtained as those induced by any of the characters $\leftindex_H\theta_{m{\m '}}$ of $H$ from Table \ref{tab:PSL} for $m{\m '}=2^um$, $0\le u\le p{-}1$. Such an induction can also be applied to the character $\leftindex_H\theta_\kappa$ of $H$ for $\kappa=(q{+}1)/3$, a value that has been excluded from ${\cal P}_\eta$ in the description of conjugacy classes. But in this case the induced character splits into $p$ character of $G$ of dimension $(q-1)$ each, which can be described as products $\theta_\kappa {\cdot}\iota_s$ for $s\in \{0,1,\ldots, p{-}1\}$, where the values of the character $\theta_\kappa$ of $G$ are those in the last line of Table \ref{tab:PGam} for $s=0$.
\smallskip

With this description we are in position to present the main result of this section.

\begin{theorem}\label{thm:charG}
A table of irreducible characters of the group $G=\PGam(2,q)$, with $q=2^p$ for prime $p\ge 5$, is as follows: \begin{table}[ht]
\scriptsize
	\centering
\begin{tabular}{|c||c|c|c|c|c|c|c|c|}
\hline \xrowht[()]{8pt}
Conjug. cl. & $[I_0]$ & $[E(0)_0]$ & $[E(1)_0]$ & $[E(\zeta(j))_0]$, $j{\in} {\cal P}_\zeta$ &  $[E(\eta(k))_0]$, $k{\in} {\cal P}_\eta$ & $[I_t]$ & $[E(0)_t]$ & $[E(1)_t]$\\ \xrowht[()]{4pt}
$\#$ classes & $1$ & $1$ & $1$ & $(q{-}2)/(2p)$ & \ \ $(q{-}2)/(2p)$ \ \ & $p{-}1$ & $p{-}1$ & $p{-}1$ \\ \hline\hline \xrowht[()]{8pt}
$\iota_s$, $0{\le}s{<}p$ & $1$ & $1$ & $1$ & $1$ & $1$ & $\gamma^{st}$ & $\gamma^{st}$ & $\gamma^{st}$ \\ \hline \xrowht[()]{6pt}
${\rm St}{\cdot}\iota_s$ & $q$ & $0$ & $-1$ & $1$ & $-1$ & $2\gamma^{st}$ & $0$ & $-\m\gamma^{st}$ \\ \hline \xrowht[()]{11pt}
$\chi_\ell$, $\ell{\in} {\cal P}_\zeta$ & $p(q{+}1)$ & $p$ & $0$ & $\sum_{u=0}^{p{-}1} \alpha(2^uj\ell)$ & $0$ & $0$ & $0$ & $0$ \\ \hline \xrowht[()]{11pt}
$\theta_m$, $m{\in} {\cal P}_\eta$  & $p(q{-}1)$ & $-\m p$ & $p$ $(3\m\m{\nmid}\m\m m)$; $-2p$ $(3{\mid}m)$ & $0$ & $-\sum_{u=0}^{p{-}1} \beta(2^ukm)$ & $0$ & $0$ & $0$ \\ \hline \xrowht[()]{6pt}
$\theta_\kappa{\cdot}\iota_s$ & $q{-}1$ & $-1$ & $1$ & $0$ & $1$ $(3\m\m{\nmid}\m\m k)$; $-2$ $(3{\mid}k)$  & $\gamma^{st}$ & $-\m\gamma^{st}$ & $\gamma^{st}$  \\ \hline
\end{tabular}
\caption{Table of irreducible characters of $G=\PGam(2,q)$, $q=2^p$, for prime $p\ge 5$.}\label{tab:PGam}
\end{table}
\end{theorem}

{\bf Proof.} The table contains $p$ characters in each of the families $\iota_s$, ${\rm St}{\cdot}\iota_s$ and $\theta_\kappa{\cdot}\iota_s$ for $s\in \{0,1,\ldots,p{-}1\}$, and $(q{-}2)/ (2p)$ characters in each of the families $\chi_\ell$ for $\ell\in {\cal P}_\zeta$ and $\theta_m$ for $m\in {\cal P}_\eta$, giving a total of $3p + (q{-}2)/2$ characters, which is equal to the number of conjugacy classes of $G = \PGam(2,q)$ for $q=2^p$ and prime $p\ge 5$. We will subsequently consider each of these five families, leaving out routine calculations based on row and column orthogonality of irreducible characters.
\smallskip

Recalling the (normal) subgroup $H\cong \PSL(2,q)$ of $G$, for every $s\in \{0,1,\dots, p{-}1\}$ the linear (and hence irreducible) character $\iota_s$ is simply a lift of a character of the cyclic group $G/H$ of order $p$ from $H$ to $G$.
\smallskip

As indicated in the overall description of characters of $G$, for every $\ell\in {\cal P}_\zeta$ the character $\chi_\ell$ of $G$ is obtained by induction from any of the $p$ characters $\leftindex_H \chi_{\ell{\m '}}$ of $H$ from Table \ref{tab:PSL} for $\ell{\m '}=2^u\ell$, $0\le u\le p{-}1$, of dimension $q+1$. Since the group $\langle\phi\rangle$ acts regularly on this set of $p$ characters, it follows from \cite[Corollary 6.3]{Whi} that the induced character $\chi_\ell$ of $G$ is irreducible, of dimension $p(q+1)$. Similarly, for every $m\in {\cal P}_\eta$ the character $\theta_m$ of $G$ arises by induction from any of the $p$ characters $\leftindex_H \theta_{m{\m '}}$ of $H$ of dimension $q-1$ for $m{\m '}=2^u m$, $0\le u\le p{-}1$. Recalling again that $\kappa = (q+1)/3$ is not contained in ${\cal P}_\eta$, this set of $p$ characters is again permuted regularly by the  group $\langle\phi\rangle$, and from \cite[Corollary 6.3]{Whi} one again concludes that the induced character $\theta_m$ of $G$ is irreducible and has dimension $p(q+1)$, for every $m\in {\cal P}_\eta$.
\smallskip

It remains to consider characters of $G$ which reduce to $G$-invariant characters of $H$, the latter being the Steinberg character and the character $\leftindex_H\theta_\kappa$. By Corollary 6.20 (or, Corollary 11.22, in a more general setting) of \cite{Isaacs}, the $G$-invariant irreducible characters of $H$ arise by restriction (to $H$) of any of the irreducible constituents of the corresponding induced character of $G$.
\smallskip

The $p$ irreducible constituents of the character of $G$ obtained by inducing the Steinberg character from $H$ to $G$ have the form ${\rm St}{\cdot}\iota_s$ for $s\in \{0,1,\ldots,p{-}1\}$, out of which ${\rm St}{\cdot} {\iota_0}$ is the actual Steinberg character of $G$. The values of of the $p$ characters ${\rm St}{\cdot}\iota_s$  on the first five conjugacy classes of Table \ref{tab:PGam} therefore coincide with the values of the Steinberg character ${\rm St}$ on $H$. In general, values of the Steinberg character for (an extension of) a linear group over a field of characteristic $2$ are known to be $0$ on conjugacy classes of elements of even order, and, up to a sign, equal to the order of a Sylow $2$-subgroup of the centraliser of an element in conjugacy classes of elements of odd order. Combining this with row orthogonality to the trivial character $\iota_0$ it follows that the values of the Steinberg character ${\rm St}{\cdot}\iota_0$ of $G$ in the three rightmost columns of Table \ref{tab:PGam} may be assumed to be equal to $2\delta,\,0,\,-\delta$ for some $\delta\in \{+1,-1\}$. But if one assumes that $\delta = -1$, the restriction of such a class function to the cyclic subgroup $C_p$ of $G$ is not a generalised character of $C_p$, which contradicts Corollary 8.12 of \cite{Isaacs}. This proves validity of the values of ${\rm St}{\cdot}\iota_s$ displayed in Table \ref{tab:PGam}.
\smallskip

The last set of characters to be dealt with arises from the character $\leftindex_H\theta_\kappa$ of $H$, stabilised by the action of the group $\langle \phi \rangle$ and hence $G$-invariant. The family of the $p$ irreducible constituents of the character of $G$ induced from $\leftindex_H \theta_\kappa$ may be assumed to have the form $\{\theta_\kappa{\cdot}{\iota_s}\}$ for $s\in \{0,1,\ldots, p{-}1\}$, of dimension $q{-}1$ each. By \cite[Corollary 6.20]{Isaacs}, each of these constituents has values $q{-}1$, $-1$, $1$, $0$ on the first four conjugacy classes of $G$ in Table \ref{tab:PGam}, respectively. By the same token, the values of each constituent on the conjugacy classes $[E(\eta(k))_0]$ for $k\in {\cal P}_\eta$ in Table \ref{tab:PGam} are $1$ if $3\nmid k$, and $-2$ if $3\mid k$, obtained by evaluating $-\beta(km)$ for $m=(q{+}1)/3$ from Table \ref{tab:PSL}.
\smallskip

To determine the values of $\theta_\kappa$ on the conjugacy classes of $G$ outside $H$, let us define a class function $\vartheta$ on conjugacy classes of $G$ by letting $\vartheta = \theta_\kappa$ on the conjugacy classes of $G$ contained in $H$, and letting $\vartheta = 1$ on the classes $[I_t]$ and $[E(1)_t]$, and $\vartheta = -1$ on the classes $[E(0)_t]$, in all cases for $t\in \{1,2,\ldots,p{-}1\}$. We will show that the restriction of $\vartheta$ to every Sylow $2$-subgroup $K$ of $G$, and also to every cyclic subgroup $L$ of $G$ of odd order, is a generalised character of $K$ and $L$, respectively.
\smallskip

Indeed, for every Sylow $2$-subgroup $K$ of $G$ the restriction $\vartheta_{\m |K}$ has values $q-1$ on the trivial class and $-1$ on the involutions of $K$, but observing that $\vartheta_{\m |K} = {\rm St}_{\m |K} - \iota_{\m |K}$ implies that $\vartheta_{\m |K}$ is a generalised character of $K$. For $L$ a non-trivial cyclic subgroup of order $d$ a divisor of $q{-}1$ one similarly obtains that $\vartheta_{\m |L} = {\rm St}_{\m |L} - \iota_{\m |L}$. For simplicity we omit the subscript on the trivial character in what follows. If $L\cong C_3$, then, letting $\nu_1$ and $\nu_2$ be the distinct non-trivial characters of $L$, for $\vartheta_{\m |L}$ with values $q-1$ on the trivial class and $1$ on elements of order $3$ one has $\vartheta_{\m |L} = c(\iota{+} \nu_1{+} \nu_2) + \iota$ for $c=(q{-}2)/3$. If $L\cong C_d$ for $d$ dividing $q{+}1$ and $d{>}3$, then the corresponding conjugacy classes of elements of order $d$ have the form $[E(\eta(k))_0]$, where $k$ satisfies the condition $d=\lcm(q{+}1, k)/k$; observe that one has $3\mid d$ if and only if $3\nmid k$. If $3\mid d$, then the restriction $\vartheta_{\m |L}$ has values $q-1$ on the trivial class and $1$ on non-trivial elements of $L$, but then one can take $\nu_1$ and $\nu_2$ to be the non-trivial characters of $L$ corresponding to the two non-trivial elements of $L\cong C_d$ of order $3$, with $\vartheta_{\m |L} = c(\iota {+}\nu_1{+}\nu_2) + \iota$ for $c=(q{-}2)/3$ as before. If $3\nmid d$, in which case $\vartheta_{\m |L}$ has values $q-1$ on the trivial class and $-2$ on non-trivial elements of $L\cong C_3$ then, letting $\Sigma$ denote the sum of all the $d$ irreducible characters of $L\cong C_d$ one may check that $\vartheta_{\m |L} = c\Sigma-2\iota$ for $c=(q{+}1)/d$. If $L\cong C_p = \{I_t\mid 0\le t\le p-1\}$, then, letting this time $\Sigma$ denote the sum of all the $p$ irreducible characters of $L$, one finds that $\vartheta_{\m |L} = c\Sigma+\iota$ for $c=(q{-}2)/p$ (which is a whole number by Fermat's Little Theorem). Similarly, for $L\cong C_{3p}$ with non-trivial elements contained in the classes $[E(1)_t]$, $1\le t\le p{-}1$, with $\Sigma$ standing for the sum of all the $3p$ irreducible characters of $L$, one may verify that $\vartheta_{\m |L} = c\Sigma+\iota$ for $c=(q{-}2)/(3p)$ the `$c$' here again being a whole number. Finally, if  $L\cong C_{2p}$ with non-trivial elements contained in the classes $[E(0)_t]$, $1\le t\le p{-}1$, one has
$\vartheta_{\m |L} = {\rm St}_{\m |L} - \iota$.
\smallskip

The Sylow $2$-subgroup $K$ with their subgroups and the cyclic subgroups $L$ considered in the previous paragraph comprise {\rm all} maximal $2$-subgroups and all maximal cyclic subgroups of $G$. These are, at the same time, the {\em elementary} subgroups of $G$, that is, those isomorphic to a direct product of a cyclic group and an $r$-group for some prime $r$. Our analysis then implies that the class function $\vartheta$ has the property that its restriction to every elementary subgroup of $G$ is a generalised character of the subgroup; one may also check that the value of the inner product $\langle \vartheta,\vartheta\rangle$ is equal to $1$. By Brauer's Characterisation Theorem \cite[Theorem 16.2]{Dor}, our class function $\vartheta$ is an irreducible character of $G$. One therefore may let $\theta_\kappa = \vartheta$, which, by multiplication by linear characters, gives rise to the family $\theta_\kappa{\cdot}\iota_s$ of $p$ irreducible characters of dimension $q{-}1$, featuring in the last row of Table \ref{tab:PGam}. \hfill $\Box$
\smallskip

\

\section{Enumeration of orientably-regular maps \\ of specified type on the group ${\bf{P\Gamma L}}({\bf{2}},{\bf{2}}^{\bf{\textit p}})$}\label{sec:enum-spec}

\
\

Let us begin by recalling the Frobenius formula \eqref{eq:Fr} from the description of {\sl Stage 2} of the second enumeration process in Section \ref{sec:strat}, by which for a finite group $G$ the number $\#({\cal A},{\cal B},{\cal C})$ of ordered triples $(x,y,z)$ of elements of $G$ with the property that $xyz=1$ and such that $x,y,z$ belong, respectively, to the conjugacy classes ${\cal A},{\cal B},{\cal C}$ in $G$, is given by
\begin{equation}\label{eq:Frob}
\#({\cal A},{\cal B},{\cal C})= \frac{|{\cal A}||{\cal B}||{\cal C}|}{|G|}\sum_{\chi} \frac{\chi({\cal A})\chi({\cal B})\chi({\cal C})}{\chi(id)} \end{equation}
with summation ranging over irreducible characters of $G$. For enumeration of ordered pairs $(x,y)$ of $G$ of given orders $m$ and $n$ such that $xy$ has order $2$, one needs to apply \eqref{eq:Fr} to all combinations of conjugacy classes ${\cal A}$, ${\cal B}$ and ${\cal C}$ of elements of order $m$, $n$ and $2$ in $G$.
\smallskip

For the group $G=\PGam(2,q)$ of our interest, with $q=2^p$ for some prime $p\ge 5$, its single conjugacy class of involutions of is contained in the unique subgroup $H=\PSL(2,q) < G$. This means that for a regular map ${\cal M}$ of type $\{m,n\}$ with $m,n\ge 3$ and such that $\Aut{\cal M} = \langle x,y\m ;\m x^m, y^n,\, (xy)^2,\, \ldots \rangle\cong G$, neither $x$ nor $y$ can be contained in $H$, for otherwise the pair $(x,y)$ would not generate $G$. By Proposition \ref{prop:p-3p}, the generator $x$ must be in one of the $3(p-1)$ conjugacy classes $C_1(t) =[I_t]$, $C_2(t)=[E(0)_t]$ and $C_3=[E(1)_t]$ of elements of order $p$, $2p$ or $3p$, respectively, for $t\in \{1,2,\ldots, p{-}1\}$. The same must hold for $y$, while the involution $xy$ must belong to the conjugacy class $[E(0)_0] = [\off(1)_0] \subset H$. In Frobenius' formula \eqref{eq:Frob} we thus may let the conjugacy classes ${\cal A}$ containing $x$ range over the conjugacy classes $C_k(t)$ for $k\in \{1,2,3\}$, while the corresponding class ${\cal B}$ for $y$ then must be one of $C_\ell(p-t)$ for $\ell\in \{1,2,3\}$ in order to have the involution $xy$ in the class ${\cal C}= [E(0)_0]$.
\smallskip

A routine substitution of these data in \eqref{eq:Frob} together with the character values for $G$ from Section \ref{sec:char} gives:
\smallskip

\begin{proposition}\label{prop:Frob} Let $A=|H|(p{-}1)(q{+}1)(q{-}2)$ and $B=|H|(p{-}1)(q{+}1)q$. The number of ordered pairs $(x,y)$ of elements of $G$ such that $\ord(xy)=2$, $\ord(x)=kp$ and $\ord(y)=\ell p$, where $1\le k\le \ell\le 3$, are as follows:
\begin{table}[ht]
\footnotesize	
\centering
\begin{tabular}{|c||c|c|c|c|c|c|}
\hline \xrowht[()]{8pt}
Type & $(p,p)$ & $(p,2p)$ & $(p,3p)$ & $(2p,2p)$ & $(2p,3p)$ & $(3p,3p)$ \\ \hline \xrowht[()]{6pt}
$\#$Pairs & $A/36$ & $B/12$    & $A/18$   & $A/4$     & $B/6$     & $A/9$  \\ \hline
\end{tabular}
\caption{The number of pairs $(x,y)\in G$ of type $\{kp,\ell p\}$, $1\le k\le \ell\le 3$.}\label{tab:Fr}
\end{table}
\end{proposition}

For every combination of types $\{kp,\ell p\}$ it remains to determine the number of {\em generating pairs} $(x,y)$  of $G$ such that $x^m=y^n=(xy)^2=1$ for $m,n\in \{p,2p,3p\}$. By the description of {\sl Stage 3} in Section \ref{sec:strat}, the number of pairwise {\em non-isomorphic} orientably-regular maps $M$ of a given type $(kp,\ell p)$ for $k,\ell\in \{1,2,3\}$ with $\Aut(M)\cong G$ is equal to the number of {\em generating pairs} $(x,y)$ of $G$ such that $\ord(x)=kp$, $\ord(y)=\ell p$ and $\ord(xy)=2$, divided by $|G|$, which is the order of $\Aut(G)$. Note that since $G$ is not generated by three involutions, all such maps are {\em chiral}.
\smallskip

To complete the enumeration for specified types as indicated, we follow a strategy similar to the one used for maps of unspecified type, but this time focusing on smooth homomorphisms and epimorphisms from the triangle groups $\Delta(2,m,n) = \langle X,Y\, :\, X^m,\,Y^n,\,(XY)^2\rangle$ to subgroups $K=\langle x,y\, :\, x^m,\,y^n,\, (xy)^2,\,\ldots\rangle$ of $G$ for $m,n\in \{p,2p,3p\}$. Specifically, for every subgroup $K\le G$ and $k,\ell\in \{1,2,3\}$ we let
\begin{align*}
   & \ho_{kp,\ell p}(K)=|\{(x,y);\, x,y\in K,\,\ord(x){=}kp,\,\ord(y){=}\ell p,\,\ord(xy){=}2\}|, \ \
   {\rm and} \\
   & \ep_{kp,\ell p}(K)=|\{(x,y);\, x,y\in K,\,\ord(x){=}kp,\,\ord(y){=}\ell p,\,\ord(xy){=}2,\,\langle x,y\rangle{=}K\}|.
\end{align*}
The value of $\ep_{kp,\ell p}(G)/|G|$ then gives the number of isomorphism classes of orientably-regular maps with automorphism group isomorphic to $\PGam(2,2^p)$, of valency $kp$ and face length $\ell p$. By symmetry one has $\ep_{kp,\ell p}(K) =\ep_{\ell p,kp}(K)$ and $\ho_{kp,\ell p}(K)=\ho_{\ell p,kp}(K)$, so one may assume that $k\le \ell$ in what follows. Also, by M\"obius inversion formula applied to the obvious fact that $\ho_{kp,\ell p}(K)=\sum_{L\le K} \ep_{kp,\ell p}(L)$ one obtains
\begin{equation}\label{eq:Mo}
    \ep_{kp,\ell p}(G)=\sum_{K\le G}\mu(K)\,\ho_{kp,\ell p}(K)=\sum_{K}\mu(K)\,\is(K)\,\ho_{kp,\ell p}(K),
\end{equation}
where, in the last sum, $K$ ranges over representatives of isomorphism classes of subgroups of $G$. It is therefore sufficient to determine the numbers $\ho_{kp,\ell p}(K)$ for each subgroup $K$ with a non-zero value of the M\"obius function. Also, the fact that $\ho_{kp,\ell p}(K)\ne 0$ implies presence of involutions and elements of order $p$ in $K$ means that the only {\em proper} subgroups $K < G$ one needs to consider for determination of $\ho_{kp,\ell p}(K)$ are $\Aff_qC_p$, $D_{2(q\pm1)}C_p$ and $C_{2p}$; for $K=G$ the numbers $\ho_{kp,\ell p}(G)$ have already been determined by characters and are summed up in Table \ref{tab:Fr} of Proposition \ref{prop:Frob}. In what follows we subsequently determine the values of $\ho_{kp,\ell p}(K)$ and, by \eqref{eq:Mo}, the values of  $\ep_{kp,\ell p}(K)$, for the six combinations of $k,\ell$ with $1\le k\le \ell\le 3$.
\medskip

{\bf Type $(p,\,p)$.} One has $\hom_{p,p}(K)=0$ for $K\in\{D_{2(q{\pm} 1)}C_p,~C_{2p}\}$, since the product of two rotations in a dihedral group cannot be a reflection and the unique involution in a cyclic group of even order cannot be a product of two elements of order $p$. For the remaining group $K=\Aff_qC_p$ we recall the equation $\is (K {>}\lfix C_{p})= 3$, established in the proof of Theorem \ref{thm:Mob} in Section \ref{sec:Mob}. The double counting of \eqref{eq:double} applied to this equation, combined with the fact that $G$ contains exactly $q{+}1$ distinct copies of $K$ (see Table \ref{tab:max} in Section \ref{sec:Mob}), then gives $\is(C_p{<}\lfix K) = q(q{-}1)/2$. This implies that $K$ contains exactly $q(q{-}1)(p{-}1)/2$ elements of order $p$; obviously each of them has the form $(\upp (\lambda,a),j)$ for $\lambda\in GF(q)^*$, $j\in\{1,2,..., p{-}1\}$ and $q/2$ elements of $a\in GF(q)$. Let $x=(\upp(\lambda,a),j)$ and  $y=(\upp(\sigma,b),j')$ be such elements of $K$ of order $p$, with $\lambda,\sigma\in GF(q)^*$ and $a,b\in \GF(q)$. For  $xy= (\upp(\lambda\sigma^{[j]},\lambda b^{[j]} +a\sigma^{-[j]}),j+j')$ to be an involution one has to take $j'=-j$ and the element $\lambda\sigma^{[j]}\in GF(q)^*$ has to be self-inverse, which implies that $\lambda\sigma^{[j]} =1$, giving a unique $\sigma$ for any  given $\lambda$. But among the $q/2$ choices for $b$ that make $y$ having order $p$ one needs to exclude $b= -a^{-[j]}$, giving $y=x^{-1}$. With $q{-}1$ choices of $\lambda$, $(q/2)(q/2{-}1)$ choices of $a,b$ and $p{-}1$ choices of $j$ one arrives at $\ho_{p,p}(\Aff_qC_p)=q(q{-}1)(q{-}2)(p{-}1)/4$, which by \eqref{eq:Mo} gives  $\ep_{p,p}(G)= |H|(p-1)(q-2)(q-8)/36$.
\smallskip

{\bf Type $(2p,\,2p)$.} For $K\in\{D_{2(q{\pm} 1)}C_p,\,C_{2p}\}$ one has $\hom_{2p,2p}(K)=0$, with details analogous to the case of type $(p,p)$. For $K=\Aff_qC_p$ one also may use ideas from the previous case. Namely, double counting \eqref{eq:double} applied to the equation $\is (K {>}\lfix C_{2p})= 1$ from the proof of Theorem \ref{thm:Mob} together with $\is(K)=q{+}1$ results in $\is(C_{2p}{<}\lfix K) = q(q-1)/2$. It follows that $K$ contains exactly $q(q{-}1)(p{-}1)/2$ elements of order $2p$, again of the form $(\upp(\lambda,a),j)$ as in the case $(p,p)$ but for a complementary set of $q/2$ values of $a\in \GF(q)$. And, as in the previous case, for a given element $x\in K$ of order $2p$ there are $q/2-1$ choices for $y\in K$ of order $2p$ such that $xy$ has order $2$. This gives $\hom_{2p,2p}(\Aff_qC_p) =q(q{-}1)(q{-}2)(p{-}1)/4$, and then, by \eqref{eq:Mo}, one finds that $\ep_{2p,2p}(G)= |H|(p-1)q(q-2)/4$.
\medskip

{\bf Type $(3p,\,3p)$.} Among the proper subgroups of $G$ one needs to consider, only the subgroup $D_{2(q{+}1)}C_p$ has an element of order $3p$. It may be checked that $D_{2(q{+}1)}C_p$ contains $2(q{+}1)(p{-}1)/3$ elements of order $3p$, all of the form $(\dia(\xi^i),\varepsilon(j))$ for $j\in \{1,2,\ldots,p{-}1\}$ and $\xi\in \GF(2^{2e})$ such that $\xi^{q{+}1}=1$, where $\varepsilon$ is the function introduced in the proof of Proposition \ref{prop:embed}. But one may check that the product of two such elements of order $3p$ is never an involution, which implies that $\hom_{3p,3p}(D_{2(q{+}1)}C_p)=0$ and hence $\ep_{3p,3p}(G)=\hom_{3p,3p}(G)= |H|(p-1)(q+1)(q-2)/9$.
\medskip

{\bf Types $(p,\,3p)$ and $(3p,\,p)$.} Taking into account an obvious similarly to the case of type $(3p,3p)$ one finds that the group $D_{2(q{+}1)}C_p$ contains $(q+1)(p-1)/3$ elements of order $p$, all of the form $(\dia(\xi^i),\varepsilon(j))$, and one may check that the product of elements of orders $p$ and $3p$ is never an  involution. As before, it follows that $\ep_{p,3p}(G)=\ep_{3p,p}(G)= |H|(p-1)(q+1)(q-2)/18$.
\medskip

{\bf Types $(2p,\,3p)$ and $(3p,\,2p)$.} As we saw in the analysis of type $(3p,3p)$, the group $D_{2(q+1)}C_p$ contains $2(q{+}1)(p{-}1)/3$ elements of order $3p$, and the same group contains $(q{+}1)(p{-}1)$ elements of order $2p$, all of the form $(\off(\xi^i),\varepsilon(j))$ for $\xi$ and $j$ as in the previous case. Thus,  for an element $x=(\dia(\xi^i), \varepsilon(j))$ of order $3p$ one has to take $-\varepsilon(j)$ in the second coordinate of a choice of $y$ of order $2p$ for $xy$ to have order $2$, and all such choices are easily seen to be valid. Since there are exactly $q{+}1$ choices for the corresponding element $y$, one has  $\ho_{2p,3p} (D_{2(q{+}1)}C_p)=\ho_{3p,2p}(D_{2(q{+}1)}C_p)=2(q{+}1)^2(p{-}1)/3$, and by \eqref{eq:Mo} one obtains $\ep_{2p,3p}(G) =\ep_{3p,2p}(G)= |H|(p-1)(q+1)(q-2)/6$.
\medskip

{\bf Types $(p,\,2p)$ and $(2p,\,p)$.} This is the most complex case and it requires analysing each of the subgroups $K\in \{\Aff_qC_p,\, D_{2(q\pm1)}C_p,\,C_{2p}\}$ separately.
\smallskip

{\bf (a)} For $K=\Aff_qC_p$, we saw in the case of type $(p,p)$ that $K$ contains $q(q{-}1)(p{-}1)/2$ elements of order $p$, each of the form $(\upp(\lambda,a),j)$; also, when considering type $(2p,2p)$ we saw that elements of order $2p$ have the same form (for different elements $a\in \GF(2^e)$, though). For a given element $x=(\upp(\lambda,a),j)$ of order $p$, if an element $x=(\upp(\sigma,b),-j)$ is such that $xy$ has order $2$, then, as in the cases $(p,p)$ and $(2p,2p)$, the element $\sigma$ is uniquely determined, and there are $q/2$ choices for $b$ for $y$ to have order $2p$. It follows that, for $K = \Aff_qC_p$ one has $\hom_{p,2p}(K) =\hom_{2p,p} (K) =q^2(q-1)(p-1)/4$.
\smallskip

{\bf (b)} If $K=D_{2(q{-}1)}C_p$, observe that $K$ contains $(q{-}1)(p{-}1)$ elements of order $p$ and the same number of elements of order $2p$, all of the form $(\dia(\lambda),j)$ and $(\off(\sigma),j)$, respectively, for $\lambda, \,\sigma\in GF(q)^*$ and $j\in\{1,2,...,p-1\}$. For an arbitrary choice of $x=(\dia(\lambda),j)$ and a corresponding choice of $y= (\off(\sigma),-j)$, one may check that $xy$ has order $2$ for any $\lambda,\, \sigma\in GF(q)^*$. Thus, for $K=D_{2(q{-}1)}C_p$ one has $\hom_{p,2p}(K)= \hom_{2p,p}(K) = (q-1)^2(p-1)$.
\smallskip

{\bf (c)} For $K=D_{2(q{+}1)}C_p$ one may use arguments analogous to the case of $D_{2(q{-}1)}C_p$. Namely, $K=D_{2(q{+}1)}C_p$ has $(q{+}1)(p{-}1)/3$ elements of order $p$, and for a given element $x\in K$ of order $p$ there are $q{+}1$ choices for $y\in K$ of order $2p$ such that the product $xy$ has order $2$. This way, for  $K=D_{2(q{+}1)}C_p$ one obtains $\hom_{p,2p}(K)=\hom_{2p,p}(K)=(q+1)^2(p-1)/3$.
\smallskip

{\bf (d)} For $K=C_{2p}$ one trivially concludes that $K$ contains exactly $p{-}1$ elements of order $p$, and for any such element $x\in K$ there is a unique $y\in K$ of order $2p$ such that $xy$ has order $2$. Thus, if $K=C_{2p}$, then $\hom_{p,2p}(K)=\hom(C_{2p})=p-1$.
\smallskip

\noindent Results in (a) -- (d) finally give $\ep_{p,2p}(G)=\ep_{2p,p}(G)= |H|(p-1)(q-2)(q-8)/12$.
\medskip

One may check that $\sum_{1\le k,\ell\le 3} \ep_{kp,\ell p}(G) = |H|(p-1)(q-1)(q-2) = |G|\, {\rm map}(G)$, where ${\rm map}(G)$ is the number of pairwise non-isomorphic orientably-regular maps of unspecified type with automorphism group isomorphic to $G=\PGam(2,2^p)$, given by Theorem \ref{thm:unspec}.
\medskip

Summarising all our findings we have proved the following statement, which is the main result of this article.

\begin{theorem}\label{MAIN}
    Let $q=2^p$ for a prime $p\ge 5$. Every orientably-regular map with automorphism group isomorphic to $\PGam(2,2^p)$ is chiral, and of type $(kp,\ell p)$ for some $k,\ell\in \{1,2,3\}$. Moreover, up to isomorphism, the number $n(\tau)$ of such orientably-regular maps of type $\tau = (kp,\ell p)$ for $k,\ell\in \{1,2,3\}$ are as listed in Table \ref{tab:fin}. 
    \begin{table}[ht]
        \footnotesize	
        \centering
        \begin{tabular}{|c||c|c|c|c|c|c|}
            \hline \xrowht[()]{8pt}
            Type $\tau$ & $(p,p)$ & each of $(p,2p)$, $(2p,p)$ & each of $(p,3p)$, $(3p,p)$   \\ \hline \xrowht[()]{6pt}
            $n(\tau)$ & $(q{-}2)(q{-}8)(p{-}1)/(36p)$ & $(q{-}2)(q{-}8)(p{-}1)/(12p)$    & $(q{+}1)(q{-}2)(p{-}1)/(18p)$     \\ \hline 
            Type $\tau$ &$(2p,2p)$ & each of $(2p,3p)$, $(3p,2p)$ & $(3p,3p)$\\ \hline
            $n(\tau)$ & $q(q{-}2)(p{-}1)/(4p)$ & $(q{+}1)(q{-}2)(p{-}1)/(6p)$ & $(q{+}1)(q{-}2)(p{-}1)/(9p)$\\ \hline
        \end{tabular}
        \caption{The number of orientably-regular maps on $\PGam(2,2^p)$ with specified types.}\label{tab:fin}
    \end{table}
\end{theorem}


\section{Invariance under map operators}\label{sec:invariance}

\

It may be useful to establish possible equivalence of the orientably-regular maps encountered in our analysis under the action of the three basic map operators of taking the dual, the Petrie dual, and forming a rotational power of a map. In our situation it is pointless to consider Petrie duality, since a Petrie dual of a chiral map cannot be orientably-regular. For duality and rotational powers, however, we can offer worthwhile conclusions.
\smallskip

Beginning with duality, observe that the orientably-regular maps of types $(kp,\ell p)$ and $(\ell p,kp)$ for fixed but distinct $k,\ell$ identified above obviously form mutually dual pairs of maps. There remains the question about possible self-duality among maps of type $(p,p)$, $(2p,2p)$ and $(3p,3p)$. Recall that an orientably-regular map ${\rm Map} (G;\,x,y)$ is positively (negatively) self-dual if there is an automorphism of $G$ of order two,  interchanging $x$ and $y$ (respectively, $x$ and $y^{-1}$). Since for $G=\PGam(2,2^p)$ we have  $G\cong \Aut(G)$, positive self-duality of ${\rm Map}(G;\,x,y)$ is equivalent to $x$ and $y$ being conjugate in $G$ by an involution. In our case, when $x=(A,j)$ for some $A\in H=\PSL(2,2^p)$ and $j\in \{1,2,\ldots,p{-}1\}$, the corresponding generator $y$ has the form $(A',-j)$ for some $A'\in H$, so that $x$ and $y$ are never conjugate in $G$ because their second coordinates are different mod $p$ (an odd prime). This shows that the maps of type $(p,p)$, $(2p,2p)$ and $(3p,3p)$ from Theorem \ref{MAIN} are never positively self-dual.
\smallskip

The analysis of eventual negative self-duality of such maps requires more attention, because by Proposition \ref{prop:p-3p} the elements $x=(A,j)$ and $y=(A',-j)$ of the same order {\em are} conjugate. We will show, however, that two such elements with product of order $2$ are never conjugate by an involution. Suppose this is not the case; as every involution of $G$ is contained in $H$, we let $(B,0) \in G$ be an involution (that is, $B^2=I$ in $H$) conjugating $x$ to $y^{-1}$, which means that $y^{-1} = (BAB^{[j]} ,0)$. Since $xy = y^{-1}x^{-1} = (BAB^{[j]}A^{-1},0)$ is also an involution, one finds that the involution $BAB^{[j]}A^{-1}$ is itself a product of two (and hence commuting) involutions $B$ and $AB^{[j]}A^{-1}$. To proceed, by Corollary \ref{cor:orders} and Proposition \ref{prop:p-3p} we may subsequently let $A$ be equal to $I$, $O$ and $E$ to cover the types $(p,p)$, $(2p,2p)$ and $(3p,3p)$, respectively.
\smallskip

If $A=I$, one has $xy = (BB^{[j]},0)$, which means that in every element of the form $(B',0)$ contained in the subgroup of $\langle x,y\rangle$, the matrix $B'$ is a product of elements of the form $(BB^{[j]})^{t}$ for various values of $t$ mod $p$. But, by the facts listed in the previous paragraph, for $A=I$ the involutions $B$ and $B^{[j]}$ commute. If $B\in H$ is a $2\times 2$ matrix with entries $(a,b)$ in its first row and $(c,a)$ in its second row (so that its trace is zero, which is equivalent to $B$ being an involution), by an elementary calculation one finds that $BB^{[j]}$ is an involution if and only if $bc^{[j]}=b^{[j]}c$, which is equivalent to $b=c$ since $\gcd([j]{-}1,[p]{-}1)=1$. Taking into account the fact that the determinant of $B$ is $1$ it follows that $B$ has the form $I+(a')$ for $a'\in \GF(2^p)$, where, for brevity, by $(a')$ we denote the $2\times 2$ matrix with all four entries equal to $a'$. But the set of $q=2^p$ such matrices forms a subgroup of $H$ isomorphic to the additive group of $\GF(2^p)$. Collecting all the obtained information one concludes that in this case the subgroup of $\langle x,y\rangle$ of elements with zero second coordinate do not generate $H$ and hence $\langle x,y\rangle\ne G$, a contradiction.
\smallskip

For $A=O$ one has $xy = (BOB^{[j]}O,0)$, again with a pair of commuting involutions $B$ and $OB^{[j]}O$. With $B$ as in the previous case, a calculation shows that $B$ and $OB^{[j]}O$ commute if and only if $bb^{[j]} = cc^{[j]}$, which is if and only if $b=c$, this time due to the fact that $\gcd([j]{+}1,[p]{-}1)=1$. The determinant condition then again yields the conclusion that $B$ has the form $I+(a')$ for $a'\in \GF(2^p)$. But such matrices are symmetric and, as a set, invariant under multiplication by $A=O$ from either side. It follows (as in the previous case) that the subgroup of $\langle x,y\rangle$ of elements with zero second coordinate forms a proper subgroup of $H$, so that $\langle x,y\rangle\ne G$, again. 
\smallskip

In the last case when $A=E$, one may check that the condition for the pair of involutions $B$ (with elements as in the previous two cases) and $EB^{[j]}E^{-1}$ to commute reads $bb^{[j]} = c(b+c)^{[j]}$. Obviously, in the conjugating involution $B$ one has $b,c\ne 0$, and with $u=cb^{-1}$ an equivalent form of the commutation condition is $u^{[j]+1}+u+1=0$. Applying the Frobenius automorphism $z\mapsto z^{[j]}$ to the last equation followed by a further multiplication by $u$ gives $u^{[2j]}u^{[j]}u + u^{[j]+1}+u = 0$, and taking the last two equations together gives $u^{[2j]} u^{[j]}u=1$. But the exponent on $u$ is $[2j]+[j]+1 = [3j]-1$, giving $u^{[3j]}=u$, which for $p\ge 5$ implies that $u\in \{0,1\}$; note that this does not apply if $p=3$. But this means that for $p\ge 5$ the equation $u^{[j]+1}+u+1=0$ has no roots whatsoever, giving yet another contradiction. It follows that none of the maps from Theorem \ref{MAIN} of the same valency and face length is negatively self-dual.
\smallskip

For considering rotational powers it is of advantage to let $z=xy$ in $G=\langle x,y\rangle\cong \PGam(2,2^p)$ and work with the equivalent presentation $G=\langle y,z;\ y^n,\, z^2,\, (yz)^m,\ldots \rangle$ for $\{m,n\} = \{ kp,\ell p\}$ and $k,\ell \in \{1,2,3\}$, which means that we may use the equivalent notation ${\rm Map}(G;\,y,z)$ for the map ${\rm Map}(G;\,x,y)$. Recalling that $y$ represents a rotation about a vertex, for every $r$ such that $\gcd(n,r)=1$ one can form from the map $M={\rm Map}(G;\,y,z)$ a new map $M^r={\rm Map}(G;y^r,z)$, called the {\em $r^{\rm th}$ rotational power} of $M$, which is obtained from $M$ by replacing, at every vertex $v$ of the underlying graph, the cyclic permutation $\pi_v$ of arcs emanating from $v$ (consistent with the orientation of the supporting surface of $M$) by the $r^{\rm th}$ power $(\pi_v)^r$ of this cyclic permutation. This means that the operation of taking a rotational power preserves the underlying graph but re-embeds it, possibly changing its carrier surface. The important property of this operation is that it preserves orientable regularity of maps; for more details we refer to \cite{Si-surv}. Such an $r$ mod $n$ is an {\em exponent} of $M$ if the maps $M$ and $M^r$ are isomorphic, which is equivalent to the existence of an automorphism of $G$ taking $y$ to $y^r$ while fixing $z$.
\smallskip

In our case of a map $M={\rm Map}(G;\,y,z)$ of valency $\ell p$ for prime $p\ge 5$ and $\ell\in \{1,2,3\}$, with automorphism group isomorphic to $G \cong \PGam(2,2^p)$, an $r^{\rm th}$ power of an element $y$ of the form $(g,j)$ has the form $(h,rj)$ and one may check that for $r$ such that $1\le r< \ell p$ and $\gcd(\ell p,r)=1$, the only possibility of $M$ and $M^r$ to be isomorphic arises when $r=1$. Equivalently, for $r\ne 1$ mod $\ell p$ the maps $M=(G;\ y,z)$ and $M^r=(G;\,y^r,z)$ are never isomorphic; in the terminology of exponents, the maps $M=(G;\ y,z)$ considered here have only the trivial exponent. Nevertheless, it is of interest to look at orbits of the operator of taking a rotational power in our situation. \smallskip

For maps of valency $p$, taking $r$ to be a primitive $(p-1)^{\rm th}$ root of $1$ mod $p$ one sees that all orbits of the operator $M\mapsto M^r$ applied to our maps $M=(G;\ y,z)$ have length $p{-}1$. For maps of valency $2p$ we may use the observation that if $r$ is a primitive $(p{-}1)^{\rm th}$ root of $1$ mod $p$, then one of $r$ and $p{-}r$ is odd and hence one of them is also a primitive $(p{-}1)^{\rm th}$ root of $1$ mod $2p$. This means that for maps $M=(G;\ y,z)$ of valency $2p$ one also obtains orbits of length $p{-}1$ under the operation of taking either the $r^{\rm th}$ or the $(p{-}r)^{\rm th}$ rotational power. In the last case when our maps $M=(G;\ y,z)$ have valency $3p$, for $p\ge 5$ and the group $U$ of units mod $3p$ one has an isomorphism $f:\ C_{p{-}1}\times C_2\to U$ of $U$ with the additive group $C_{p{-}1}\times C_2$. Writing now $C_{p{-}1}\times C_2$ in the form $\langle (r,1),(p{-}r,1)\rangle$ for a generator $r$ of $C_{p{-}1}$, it follows that the operators $M\mapsto M^s$ and $M\mapsto M^t$ for $s=f(r,1)$ and $t=f(p{-}r,1)$ induce orbits of length $p{-}1$ on the set of maps $M=(G;\ y,z)$ of valency $3p$. This all yields the following consequence. 
\smallskip

\begin{corollary}\label{Cor-MAIN} The maps of type $(kp,\ell p)$ and $(\ell p,kp)$ for $k\ne \ell$ listed in  Table \ref{tab:fin} form mutually dual pairs, while none of the remaining maps in Table \ref{tab:fin} are positively or negatively self-dual. In all cases, for every fixed $\ell\in \{1,2,3\}$, the maps of valency $\ell p$ from Table \ref{tab:fin} form orbits of length $p{-}1$ under suitable rotational powers induced by primitive roots of unity {\rm mod} $p$.
\end{corollary}


\section{Remarks}\label{sec:rem}

\

We left out of our consideration of orientably-regular maps with automorphism group $\PGam(2,2^p)$ the two smallest primes. For $p=2$ the group $\PGam(2,4)$ is isomorphic to the symmetric group $S_5$ of degree $5$, and also to ${\rm PGL}(2,5)$, and the corresponding maps are known. For example, by the enumeration results of \cite{Sah} for fractional linear groups, the group ${\rm PGL}(2,5)$ supports, up to isomorphism, a total of $7$ orientably-regular maps: one map for each of the types $(4,5)$, $(4,6)$ and $(5,6)$, their duals, and one self-dual map of type $(6,6)$. In contrast with the remaining orientably-regular maps encountered in this article, all these $7$ maps are reflexible, that is, isomorphic to their mirror image.
\smallskip

We have also excluded the case $p=3$ from consideration because the group $\PGam(2,2^3)$ contains two conjugacy classes of subgroups of order $3$ and three conjugacy classes of subgroups of order $9$, which destroys   uniqueness of conjugacy classes seen in Proposition \ref{prop:p-3p} for $p\ge 5$. Types of orientably-regular maps supported by $\PGam(2,2^3)$ are also restricted to $(3k,3\ell)$ for $k,\ell\in \{1,2,3\}$, but here we may exclude the spherical type $(3,3)$ and the toroidal types $(3,6)$ and $(6,3)$ because the automorphisms groups of orientably-regular maps of such types are soluble.  Calculations by MAGMA show that among the total of $28$ pairwise non-isomorphic orientably-regular chiral maps with automorphism group isomorphic to $\PGam(2,2^3)$, the distribution of the remaining types is as given in Table \ref{tab:8}, coinciding with entries arising from Table \ref{tab:fin} for $p=3$. As regards map operators, there is no kind of self-duality in any of the maps of type $(6,6)$ and $(9,9)$. For rotational powers, maps of valency $3$ and $6$ form an orbit of size $p{-}1=2$ each,  but those of valency $9$ give rise to two orbits of size $6$, due to the fact that there is a $6^{\rm th}$ primitive root of unity mod $9$ (unlike the situation in maps of valency $3p$ for $p\ge 5$ from Theorem \ref{MAIN}).
\begin{table}[ht]
    \footnotesize	
    \centering
    \begin{tabular}{|c||c|c|c|c|c|}
        \hline \xrowht[()]{8pt}
            Type $\tau$ & each of $(3,9)$, $(9,3)$ & $(6,6)$ & each of $(6,9)$, $(9,6)$ & $(9,9)$& {Total} \\ \hline
            $n(\tau)$ & 2 & 8 & 6 & 4 & 28\\ \hline
        \end{tabular}
    \caption{The number of orientably-regular maps on $\PGam(2,8)$ of specified type.}\label{tab:8}
\end{table}

It may be of interest to observe that, for $q=2^p$ with $p\ge 5$, the formula of Theorem \ref{thm:unspec} together with the results of \cite[p. 302]{DJ} imply that the numbers of isomorphism classes of orientably-regular maps $M$ with $\Aut(M)\cong \PGam(2,q)$ is $(p{-}1)$-times the number of such isomorphism classes of maps but with $\Aut(M)\cong \PSL(2,q)$. This makes one wonder if there a simpler explanation for this phenomenon, rather than the actual enumeration.
\smallskip

Also, the enumeration process, which makes part of the proof of Theorem \ref{thm:unspec} and is based on Table \ref{tab:X}, seems to suggest that there may be a way to link the M\"obius functions of a group and of its split extension by a cyclic group of prime order. If true, it would be desirable to establish such a link.
\smallskip

Finally, a further challenge may be to try to generalise Theorem \ref{thm:unspec} to orientably-regular maps $M$ of unspecified type with $\Aut(M)\cong \PGam(2,q)$ for $q=2^e$ for arbitrary odd $e\ge 1$; the case of even $e$ may provide unexpected obstacles due to the fact that in this situation the group $\PGam(2,2^e)$ does not embed in $\PGam(2,2^{2e})$, by Proposition \ref{prop:embed}.
\smallskip
\bigskip

{\bf Acknowledgement.} The three authors from China were partially supported by the National Natural Science Foundation of China (12331013, 12271024, 12425111) and the 111 Project of China (B16002); Chihao Wang particularly thanks Professor Jiyong Chen for a useful conversation on the enumeration problem for specific types of maps.

The two authors from Slovakia acknowledge support for this research by the APVV Research Grants 22-0005 and 23-0076, and by the VEGA Research Grants 1/0069/23 and 1/0011/25.

\end{document}